\documentclass[a4paper, 12pt, reqno]{amsart}

\usepackage[colorinlistoftodos]{todonotes}
\usepackage{amsmath, amsthm, amssymb, amsfonts, fullpage}
\usepackage[maxbibnames=99,backend=biber,bibencoding=utf8,style=numeric,giveninits=true,url=false,doi=false,isbn=false]{biblatex}
\usepackage{mathtools}
\mathtoolsset{showonlyrefs}
\usepackage{graphicx, subfigure, color,calc}
\numberwithin{equation}{section}

\title[Critical Behavior of Non-Intersecting Brownian Motions]{
Critical Behavior of Non-Intersecting Brownian Motions}

\author[]{Tom Claeys$^*$}
\address{$^*$Institut de Recherche en Math\'ematique et Physique, Universit\'e catholique de Louvain, Chemin du Cyclotron
2, B-1348 Louvain-La-Neuve, Belgium}
\email{tom.claeys@uclouvain.be}

\author[]{Thorsten Neuschel$^\dag$}
\address{$^\dag$Department of Mathematics, Bielefeld University, Germany}
\email{thorsten.neuschel@math.uni-bielefeld.de}

\author[]{Martin Venker$^\ddag$}
\address{$^\ddag$Department of Mathematics, Bielefeld University,
Germany}
\email{mvenker@math.uni-bielefeld.de}
\date{\today}

\newcommand{\R}{\mathbb{R}}

\DeclareMathOperator{\Tr}{Tr}

\newcommand{\Ai}{{\rm Ai}}
\theoremstyle{plain}
\newtheorem{thm}{Theorem}[section]

\newtheorem{lemma}[thm]{Lemma}

\newtheorem{prop}[thm]{Proposition}
\theoremstyle{remark}
\newtheorem{rmk}[thm]{Remark}

\newcommand{\lb}{\left(}
\newcommand{\rb}{\right)}

\renewcommand{\O}{\mathcal O}
\newcommand{\lv}{\lvert}
\newcommand{\rv}{\rvert}
\renewcommand{\k}{\kappa}
\renewcommand{\t}{\tau}

\newcommand{\s}{\sigma}
\renewcommand{\b}{\beta}
\renewcommand{\epsilon}{\varepsilon}
\newcommand{\e}{\epsilon}
\newcommand{\w}{\omega}
\newcommand{\z}{\zeta}
\renewcommand{\hat}{\widehat}
\renewcommand{\P}{\mathbb P}

\renewcommand{\d}{\delta}

\newcommand{\x}{\mathbf{x}}
\newcommand{\A}{\mathcal A}

\usepackage[normalem]{ulem}

\begin{document}

\begin{abstract}
We study $n$ non-intersecting Brownian motions corresponding to initial configurations which have a vanishing density in the large $n$ limit at an interior point of the support.
It is understood that the point of vanishing can propagate up to a critical time, and we investigate the nature of the microscopic space-time correlations near the critical point and critical time.
We show that they are
described either by the Pearcey process or by the Airy line ensemble, depending on whether a simple integral related to the initial configuration vanishes or not. Since the Airy line ensemble typically arises near edge points of the macroscopic density, its appearance in the interior of the spectrum is surprising. We explain this phenomenon by showing that, even though there is no gap of macroscopic size near the critical point, there is with high probability a gap of mesoscopic size. Moreover, we identify a path which follows the Airy$_2$ process.
\end{abstract}

%\keywords{Random matrices, sine kernel, universality, Dyson's Brownian motion}

\maketitle

\section{Introduction and Main Results}
The subject of non-intersecting Brownian motions has received a lot of attention over the last 15 years, due to the connection to random matrix theory and as continuum analog of discrete interacting particle systems.

They were studied presumably first by  Dyson in 1962 \cite{Dyson} as a physically motivated dynamic model of eigenvalues of random matrices. More precisely, he considered a Brownian motion on the space of $n\times n$ Hermitian matrices $(M(t))_{t\geq0}$, starting at time $t=0$ with an arbitrary Hermitian matrix $M(0)$, and 
he was interested in the process $(X(t))_{t\geq0}$ of the $n$ (real) eigenvalues $X_1(t)\leq X_2(t)\leq\dots\leq X_n(t)$ of $\lb\frac{1}{\sqrt n}M(t)\rb_{t\geq0}$, 
where we state the rescaling factors $n^{-1/2}$ here for later convenience. In the  language of stochastic analysis, Dyson's main finding was that the process $(X(t))_{t\geq 0}$ solves the system of stochastic differential equations (SDE)
\begin{align}
dX_j(t)=\frac{1}{\sqrt n}dB_j(t)+\frac{1}{n}\sum_{k\not=j}\frac{1}{X_j(t)-X_k(t)}dt,\label{SDE}\quad j=1,\ldots,n,
\end{align}
where $(B_1(t)_{t\geq0},\dots,(B_n(t))_{t\geq0}$ are $n$ independent Brownian motions.
 Later it was realized that this process can be obtained by conditioning $n$ independent Brownian motions (with diffusion coefficients $n^{-1/2}$) not to intersect for all times \cite{Grabiner}, thereby establishing the term \textit{non-intersecting/non-colliding Brownian motion{s}}. In the following, we will {use} the name non-intersecting Brownian motions (NIBM) instead of the also popular term \textit{Dyson's Brownian motions}, as the latter one often is used for a whole one-parameter family of {processes} obtained by replacing the factor $1/n$ in  \eqref{SDE} by $\frac{\b}{2n}$, where $\b>0$ is arbitrary. For the values $\b=1,4$, Dyson's Brownian motion{s} can be realized as eigenvalue process{es} of Brownian motion{s} in the spaces of real-symmetric and quaternionic-selfdual matrices, respectively. Only for the value $\b=2$, considered in this paper, the process admits an interpretation of $n$ independent processes conditioned not to intersect.

{A Hermitian Brownian motion is a natural dynamic version of the so-called deformed Gaussian unitary ensemble (deformed GUE), a well-known model of Hermitian random matrices, defined as the matrix distribution with density (in $M$) proportional to 
\begin{align}
e^{-\frac n2\Tr(M(0)-M)^2}, \label{deformed_GUE}
\end{align}
where $M(0)$ is a given deterministic Hermitian matrix. The eigenvalues of the deformed GUE and NIBM show in general a much richer behavior than the eigenvalues of the  classical GUE, which is obtained for $M(0)=0$. 
}

In the quest for discovering universal structures hidden in physical and mathematical models, here one is interested in the local correlations of the paths (or trajectories) as their number grows to infinity. The term ``local correlations'' refers to correlations on a microscopic scale, i.e.~on a scale on which eigenvalues can be observed individually. {For instance, in typically considered bulk situations} {a local rescaling of NIBM leads for a large number of paths to  the sine process.} This is a time-dependent determinantal point process given in terms of the extended sine kernel, and this limit can be observed universally with regards to all initial conditions leading to such bulk situations \cite{Johansson01,Shcherbinabulk, Erdosetal1,LeeSchnellietal,ClaeysNeuschelVenker}. {Note that sine-kernel statistics also appear in a large variety of interacting particle systems, a fact commonly referred to as universality. Instead of this (model) universality, we will in this paper focus on universality in the initial conditions.}

{Spectral edge situations have been studied extensively as well. Here, {whenever the limiting macroscopic density has a square root vanishing followed by a macroscopic gap}, the so-called Airy line ensemble typically arises \cite{Johansson03,Shcherbinaedge,CapitainePeche,LeeSchnelli} and describes the statistics of the largest or smallest eigenvalues in a bulk of eigenvalues. }

{The situation of two bulks of eigenvalues merging at a certain time is less well-understood. Whenever the limiting macroscopic density has a cusp, {one expects} to observe the 
Pearcey process \cite{TracyWidom2}. There are however cases in which a cusp singularity does not lead to Pearcey statistics  but e.g.~to cusp-Airy processes \cite{DJM} or $r$-Airy processes \cite{ADvM} which are related to outliers at the edge.  A mere touching (but not merging) of two bulks at one particular time can be observed in certain cases for non-intersecting Brownian bridges and is supposed to lead to the tacnode process (see \cite{Johansson13,AJvM} and references therein). The Pearcey, cusp-Airy and $r$-Airy processes have mainly been found for special initial conditions of NIBM (or Brownian bridge models) or special discrete interacting particle systems. Classification results on Airy and Pearcey universality  have been given in \cite{CapitainePeche,Erdosetal} for the eigenvalues of the deformed GUE \eqref{deformed_GUE}, which may be regarded as one-time distribution of NIBM. Because of the above mentioned variety of limiting processes, a classification of the initial conditions producing Pearcey statistics in the large $n$ limit is necessarily complicated.}\\ 

In view of these complications, we will in this paper study concrete situations in which a merging of bulks naturally appears. More precisely, we consider initial configurations whose limiting density $\psi$ vanishes at an isolated point $x^*$. As we will see, a sufficiently fast vanishing of the density at $x^*$ leads to the merging of two bulks at some critical time $t_{\rm cr}>0$ and some critical point $x^*(t_{\rm cr})$. Interestingly, this natural situation is not covered by the previous classification results of \cite{CapitainePeche,Erdosetal}. Moreover, two different limiting processes can be found at the merging:  if both bulks are  ``in balance'', e.g.~in symmetric settings, then the Pearcey process emerges for $n\to\infty$. If one bulk ``dominates the other'', we surprisingly find the Airy line ensemble in the interior of the spectrum. The condition that determines whether the Pearcey process or the Airy line ensemble appears, is non-local and remarkably simple: we have the Pearcey process if the integral
	\begin{align}
	\int\frac{\psi(s)ds}{(x^*-s)^3}\label{integral}
	\end{align} vanishes, and the {Airy} line ensemble otherwise.

{Let us now proceed towards a precise statement of our results.}
Let 
\begin{align}
\mu_n:=\frac1n\sum_{j=1}^n\d_{X_j(0)}
\end{align}
denote the empirical measure of the initial points, which we will assume to be deterministic.
\paragraph{\bf Assumption 1.} The support of all empirical measures $\mu_n$ is contained in a fixed (independent of $n$) bounded set, and \(\mu_n\) converges weakly to a probability measure \(\mu\) as \(n\to\infty\). We assume that $\mu$ has a density $\psi$ w.r.t.~the Lebesgue measure, which is continuous as a function on the support and such that there is a point $x^*\in\R$ with
\begin{equation}\label{eq:psilocal}\psi(x)\sim c\vert x-x^{*} \vert^{\kappa}\quad, x\to x^*, \end{equation}
with \(\kappa>2\) and some constant $c>0$. {We will call $x^*$ a \textit{critical point}.}

{Before commenting on the role of $\k$, we recall the well-known fact of free probability theory that} for any fixed time $t>0$, the empirical measure 
\begin{align}
\mu_{n,t}:=\frac1n\sum_{j=1}^n\d_{X_j(t)}\label{pp}
\end{align} 
converges as $n\to\infty$ weakly almost surely to a probability measure $\mu(t)$ given by the {{\em free additive convolution}} of $\mu$ and the semicircle distribution with support $[-2\sqrt t,2\sqrt t]$ \cite[Chapter 5]{AGZ}. This measure has again compact support and a continuous density $\psi_t$ for $t>0$ \cite{Biane}.

\begin{rmk}[On the role of $\k$]\label{rmk_kappa}\leavevmode
	\begin{enumerate}
		\item If $\kappa\leq 1$, then for all $t>0$ sufficiently small, the density \(\psi_t\) of the evolved measure $\mu(t)$ is positive at $x^*$, meaning that the zero of $\psi$ is instantly removed. {There are thus no separate bulks of eigenvalues in this case.} It is shown in \cite{ClaeysNeuschelVenker} that for $0\leq\k<1$, under certain natural assumptions on the initial points, sine kernel correlations arise close to $x^*$ {already} for very short times $$t=t_n=\lb\frac{\log(n)^{1+\rho}}{n}\rb^{\frac{1-\k}{1+\k}},$$ for any $\rho>0$. 
		\item If $\kappa>1$, then for all sufficiently small $t>0$, there is a point $x^*(t)$ near $x^*$ such that $\psi_t(x^*(t))=0$. Thus in this case the zero in the support of $\psi$  {persists} and is propagated {for some strictly positive time along a specific path  which separates two bulks of eigenvalues.} See Figure \ref{figure1} {below} for a visualization of a realization of NIBM in this case. {In such situations}, it is not reasonable to expect sine kernel universality for small values of $t$. It can be expected though, that for small $t$ there are asymptotically no (non-trivial) local correlations of $X(t)$ in the vicinity of the critical point $x^*$ (or rather a suitably evolved critical point $x^*(t)$), whereas sine kernel universality should be observable only for times \(t\) beyond some particular critical time $t_{\rm cr}>0$. Both expectations have been confirmed (to some extent) in \cite{ClaeysNeuschelVenker}. In the present work, we investigate the correlations exactly at the critical time, i.e. we focus precisely on the location and the time at which the transition from deterministic behaviour to random matrix behaviour takes place.
		\item {There appears to be a genuine difference between the cases \(1<\kappa\leq 2\) and \(\kappa>2\) regarding universality in the initial conditions, see Remark \ref{remark:Airy} below.}
	\end{enumerate}
\end{rmk}
When the process evolves in time, in general $x^*$ will not be a zero of the density $\psi_t$ for any $t>0$, even in the case of a zero persisting for some time.
{In order to keep track of the critical point as time evolves, we  consider a deterministic time evolution $(x^*(t))_t$ of the initial critical point $x^*$.} 
{It can be  interpreted as the typical evolution of a path, located at $x^*$ at time $0$, until it gets included in the support of $\psi_t$ at some critical time $t_{\rm cr}(x^*)$.}

To define this evolution, we recall a description of the {evolved} measure $\mu(t)$ and its density $\psi_t$ due to Biane \cite{Biane}: We define for \(x \in \mathbb{R}\) and \(t>0\),
\begin{align}
y_{t,\mu}(x):=\inf\left\{y>0 : \int\frac{d\mu(s)}{(x-s)^2+y^2} \leq \frac{1}{t} \right\}\label{def:y}
\end{align} 
and we set
\begin{align}
\Phi_t(x):=H_{t,\mu}\left(x+i y_{t,\mu}(x)\right),\label{def:Phi}
\end{align}
where
\begin{align}
H_{t,\mu}(z):=z+tG_{\mu}(z),\quad\text{ and }\quad G_{\mu}(z):=\int \frac{\mu(ds) }{z-s} \label{def:Stieltjes}
\end{align}
is the Stieltjes transform of $\mu$. For fixed \(t >0\), $\Phi_t$ is a bijection from \(\mathbb{R}\) to \(\mathbb{R}\).   The functions $y_{t,\mu}$ and $\Phi_t$ make it possible to express the density $\psi_t$ of  $\mu(t)$ (which exists for $t>0$ for any $\mu$) as 
\begin{align}
\psi_t(\Phi_t(x))=\frac{y_{t,\mu}(x)}{\pi t}.\label{eq:densityPhi_t}
\end{align}
This representation is convenient for us as the right-hand side does not depend on $\mu(t)$ but on $\mu$ alone, which is easier to describe in terms of our assumptions. Moreover, it suggests to define  the evolution of a point $x\in\R$ as $$x(t):=\Phi_t(x).$$ To connect this analytic definition to  {NIBM}, we note that $x(t)$ solves the {\em linearized mean field equation} 
\begin{align}
x(t)=x+t\,\textup{P.V.}\int \frac{1}{x(t)-s}\mu(t)(ds),\label{mean_field_evolution}
\end{align}
where the principal value integral is the Hilbert transform of $\mu(t)$ at $x(t)$. {Recall that $\mu(t)$ is the a.s.~weak limit of  $\mu_{n,t}$ from \eqref{pp}.}
From a heuristic point of view, equation \eqref{mean_field_evolution} can be derived from \eqref{SDE} by considering a particle $X_j(t)$ with $X_j(0)=x$ and replacing the drift term 
$$\int_{\not=}\frac{1}{X_j(t)-s}\mu_{n,t}(ds) $$
 (with $\int_{\not=}$ understood as excluding the $j$-th particle) by its natural limit $$\textup{P.V.}\int \frac{1}{x(t)-s}\mu(t)(ds).$$

To understand the definition of the evolution $t\mapsto x(t)$, we observe that by \eqref{eq:densityPhi_t} and \eqref{def:y} $\psi_t(x(t))=0$ if and only if
\begin{align*}
\int\frac{\mu(ds)}{(x-s)^2}\leq\frac1t.
\end{align*}
This explains in particular the distinction between the cases  $\k\leq1$ {and} $\k>1$ in Remark \ref{rmk_kappa}. Moreover, defining for a general $x\in\R$ the \textit{critical time} \(t_{\rm cr}=t_{\rm cr}(x)\) by
\begin{align}
t_{\rm cr}(x):=\left(\int\frac{\mu(ds)}{(x-s)^2}\right)^{-1} = -\frac{1}{G_{\mu}'(x)}\label{def:t_cr}
\end{align}
with the convention $t_{\rm cr}=0$ if the integral is $\infty$, we have that
{\[\psi_t (x(t))\begin{cases}
=0,&\quad\text{ if }\ 0< t \leq t_{\rm cr}(x),\\
>0,& \quad  \text{ if }\ t>t_{\rm cr}(x).
\end{cases} \]}
Up to the critical time, {$x(t)$} is linear in $t$, which can be seen readily from \eqref{def:Phi} and \eqref{def:Stieltjes}. 

Returning to the critical point $x^*$ (in the sense of \eqref{eq:psilocal}) and $\k>2$, we have the following:  $x^*(t)$ is critical for any $t$ up to the critical time 	{$t_{\rm cr}:=t_{\rm cr}(x^*)>0$} (meaning that $\psi_t(x^*(t))=0$ for $t\leq t_{\rm cr}$) and the transition from deterministic to sine kernel statistics is expected to occur precisely around $x^*(t_{\rm cr})$ at time $t_{\rm cr}$.

Some analytic aspects of the density $\psi_{t_{\rm cr}}(x)$ for $x$ close to $x^*(t_{\rm cr})$ have been studied in \cite{ClaeysKuijlaarsLiechtyWang} for the special case of $\k$ being an even integer, where it  {was} shown that the behavior of  \(\psi_{t_{\rm cr}}\) at the critical time in a neighborhood of \(x_{t_{\rm cr}}^{*}\) depends on the sign of \(G_{\mu}''(x^{*})\), which is \eqref{integral} up to a factor 2. More precisely, for integer-valued $\k\geq4$, we have  
\begin{equation} \label{eq:psilocalcritthmII}
\psi_{t_{\rm cr}}(x) =
\begin{cases}  
\frac{\sqrt{2}}{\pi t_{\rm cr}^{3/2} G_\mu''(x^*)^{1/2}}
\left| x-x^*(t_{\rm cr}) \right|^{1/2} (1 + o(1)), & 
\text{as } x \to x^*(t_{\rm cr})_-, \quad G_{\mu}''(x^{*})>0,\\[15pt]  
\mathcal O\left(|x-x^*(t_{\rm cr})|^{\frac{\kappa-1}{2}}\right), &\text{as } x \to x^*(t_{\rm cr})_+,\quad G_{\mu}''(x^{*})>0,\\[15pt]
\frac{\sqrt{3}}{2\pi t_{\rm cr}^{4/3} G_\mu'''(x^*)^{1/3}}
\left| x-x^*(t_{\rm cr}) \right|^{1/3} (1 + o(1)), & 
\text{as } x \to x^*(t_{\rm cr}), \quad G_{\mu}''(x^{*})=0.
\end{cases}
\end{equation}
 The case $G_\mu''(x^*)<0$ is analogous to $G_\mu''(x^*)>0$ with obvious changes.  See Figure \ref{Criticalden} for {visualizations of the two fundamentally different local behaviors of the density.}

For  $G_{\mu}''(x^{*})=0$, the appearance of a cubic root zero in the asymptotics \eqref{eq:psilocalcritthmII} suggests  Pearcey type fluctuations around the cusp $x^*(t_{\rm cr})$. For $G_{\mu}''(x^{*})\not=0$, the situation is a priori not clear because we have square root behaviour as $x^*(t_{\rm cr})$ is approached from one side, but a different order of vanishing as it is approached from the other side. One may argue on a heuristic level that the non-square root decay is faster, {giving} a less strong accumulation of particles compared to behavior of the particles on the side of the square root vanishing. This should result in larger distances between particles on the non-square root side than 
the distances predicted by the square root behavior and hence {we might expect the Airy line ensemble making its appearance in this unusual situation.}
Below we put these heuristics for all $\k>2$ (not necessarily integers) on firm ground.\\

\begin{figure}[h]
	\centering
	\begin{minipage}[t]{0.45\linewidth}
		\centering
		\includegraphics[scale=0.33]{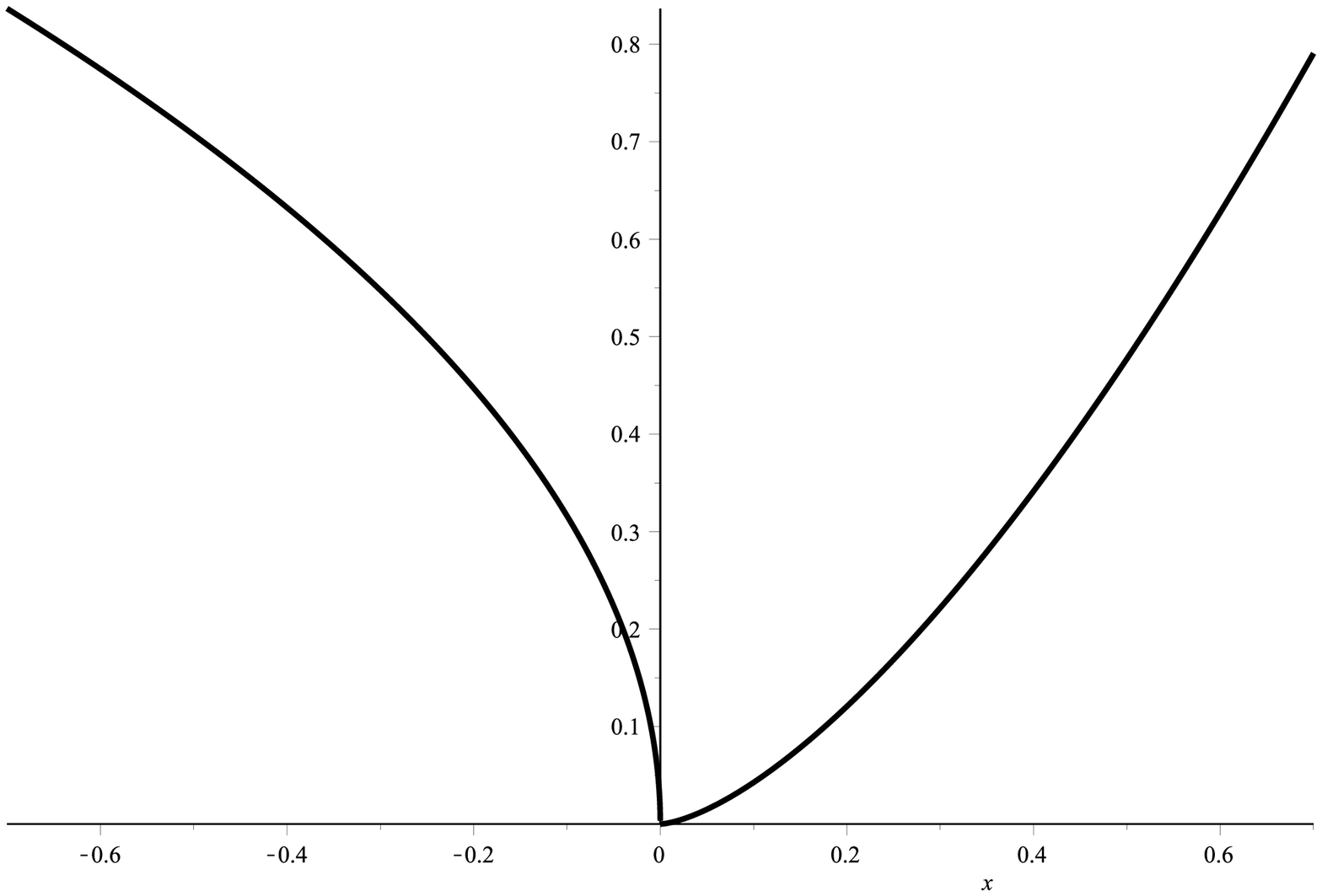}
	\end{minipage}%
	\hfill
	\begin{minipage}[t]{0.45\linewidth}
		\centering
		\includegraphics[scale=0.33]{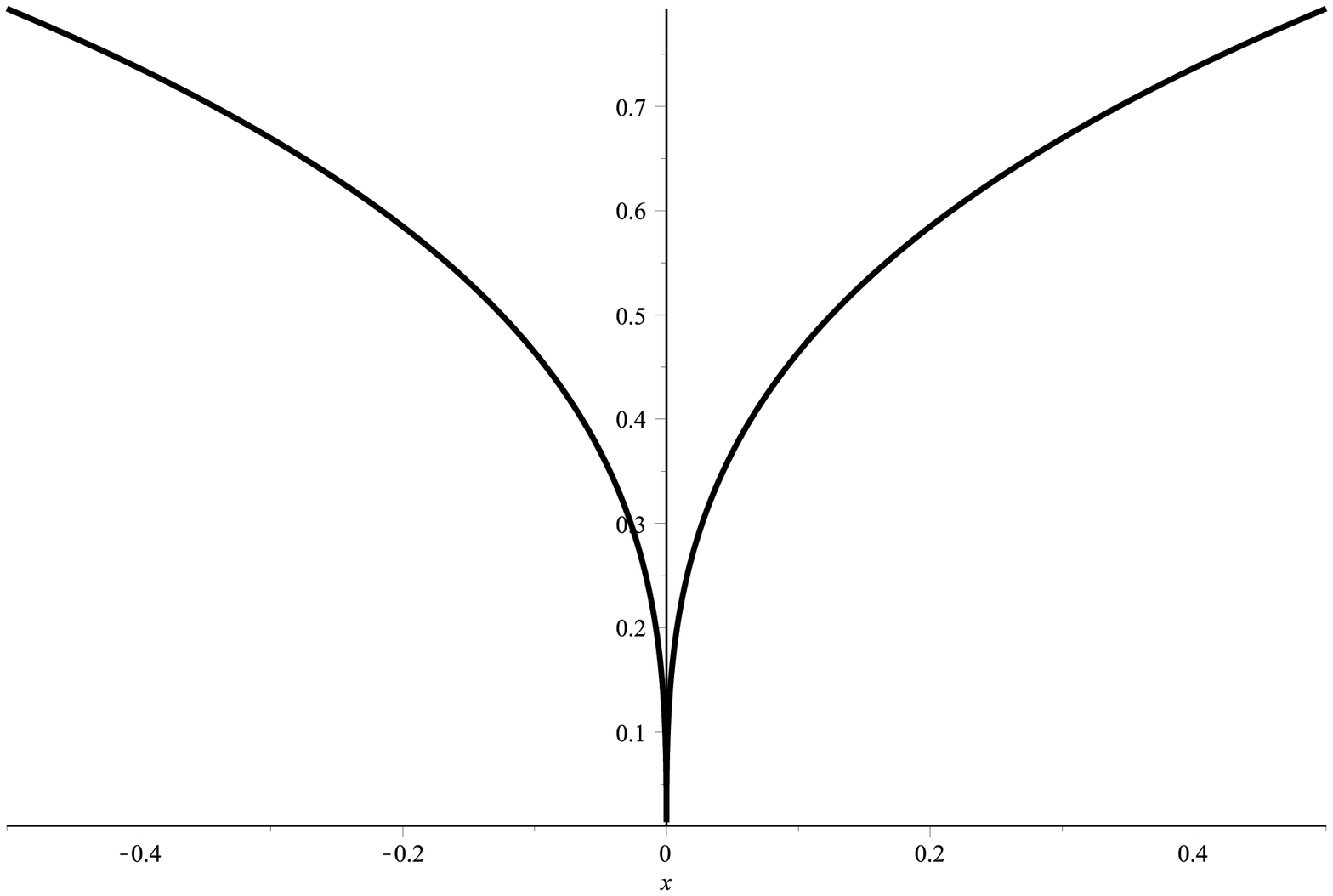} 
	\end{minipage} %\vspace{-5em}
	\label{Criticalden}
	\caption{Typical local behavior of the density at the critical point $x^*(t_{\rm cr})=0$ and at the critical time, for $\kappa=4$ and $G''_\mu(x^*)>0$ on the left, and for $G''_\mu(x^*)=0$ on the right. }
\end{figure}

{It is well-known that the space-time correlation functions  {of} NIBM have a determinantal structure. We refer the reader to Appendix \ref{appendix:correlations} for the definition and a review of relevant properties and expressions of the correlation functions. Here, we restrict ourselves to stating that the space-time correlation functions 
$\rho^{(n)}_{t_1,\dots,t_k}(\x^{(1)}_{m_1},\dots,\x^{(k)}_{m_k})$ {with $\x^{(j)}_{m_j}:=(x^{(j)}_1,\dots,x^{(j)}_{m_j})$}, corresponding to $k$ times $t_1,\ldots, t_k$ and $m_j$ positions $x_1^{(j)},\ldots, x_{m_j}^{(j)}$ at time $t_j$, can be written as
\begin{align}\label{determinantality}
\rho^{(n)}_{t_1,\dots,t_k}(\x^{(1)}_{m_1},\dots,\x^{(k)}_{m_k})=\det\lb\Big[ K_{n,t_i,t_j}\lb x^{(i)}_p,x^{(j)}_q\rb\Big]_{\substack{1\leq p\leq m_i\\1\leq q\leq m_j}}\rb_{1\leq i,j\leq k},
\end{align}
where the matrix on the right-hand side of \eqref{determinantality} is of size $\lb\sum_{l=1}^k m_l\rb^2$ and the correlation kernel $K_{n,s,t}(x,y)$ is given by \cite{TracyWidom2,KatoriTanemura}
\begin{multline}\label{def:Kn}
K_{n,s,t}(x,y)\\
:=\frac{n}{(2\pi i)^2 \sqrt{st}} \int_{x_0+i\R} dz \int_{\Gamma} dw ~ \frac{\exp\lb\frac{n}{2t} \left[(z-y)^2 +2t g_{\mu_n}(z)\right]-\frac{n}{2s} \left[(w-x)^2 +2s g_{\mu_n}(w)\right]\rb}{z-w}\\
-1(s>t)\frac{\sqrt{n}}{\sqrt{2\pi(s-t)}}\exp\lb-\frac{n}{2(s-t)}(x-y)^2\rb.
\end{multline}
Here $g_{\mu_n}$ is given by \[g_{\mu_n}(z) = \int \log(z-s) d\mu_n(s),\]
defined by choosing the principal branch of the logarithm,
\(\Gamma\) is a positively oriented closed curve in the complex \(w\)-plane that encircles all the initial particles \(x_1^{(0)},\ldots, x_n^{(0)}\) (or a finite union of disjoint closed curves such that each initial particle has winding number $1$) and does not intersect with $x_0+i\R$, where $x_0\in\R$ is arbitrary apart from the condition on non-intersection with $\Gamma$.
}

In order to evaluate the asymptotic behavior of the double integral in the above formula for the kernel for large values of \(n\), on a technical level we will require a decent control of derivatives of $g_{\mu_n}$ locally. To this end, we make the following assumptions.

\medskip

\paragraph{\bf Assumption 2.}
We assume that there exist constants $M>0, n_0>0$ such that
\begin{equation}\label{eq:assumptionFnF}
|F_n(x)-F(x)|\leq \frac{M}{n}
\end{equation}
for all $n\geq n_0$ and for all $x\in\mathbb R$, where $F_n$ and $F$ denote the distribution functions of $\mu_n$ and $\mu$, respectively.

We can interpret this assumption as a condition on the rigidity of the eigenvalues at the initial time \(t=0\), which is satisfied, for instance, if for any $j$, the $j$-th smallest eigenvalue $X_j(0)$ lies sufficiently close to the (suitably defined) $j/n$-quantile of $\mu$.  Although it seems possible to relax this assumption for $x$ not too close to $x^*$, we choose not to do this to avoid further technical complications. It is important to point out, however, that the assumption can not be removed entirely. Indeed, the weak convergence of $\mu_n$ to $\mu$ does not imply any quantitive comparability of $\mu_n$ and $\mu$ on smaller scales which is needed when considering space-time correlations localized around $(t_{\rm cr},x^*(t_{\rm cr}))$ (cf.~\cite[Theorem 1.4]{ClaeysNeuschelVenker}).

\medskip

\paragraph{\bf Assumption 3.}
There exist constants $m>0$ and $n_0\in\mathbb N$ such that the interval $\left[x^*-mn^{-\frac{1}{\kappa+1}}, x^*+mn^{-\frac{1}{\kappa+1}}\right]$ has no intersection with the support of $\mu_n$ for $n\geq n_0$.

\medskip

{This third assumption in particular forbids the existence of (isolated) initial particles that lie right in the middle of the two bulks. Such particles might change the limiting behavior at criticality much like the separation of $r$ particles at the edge induces a change from Airy kernel to $r$-Airy kernel statistics \cite{ADvM}. In this sense, we believe an assumption like Assumption 3 to be necessary. It} is not very restrictive in the sense that the distances between consecutive quantiles of the limiting distribution $\mu$ near the critical point \(x^* \)  {are} precisely of order \(n^{-\frac{1}{\kappa+1}}\). {As a simple example of an initial configuration satisfying Assumptions 1-3, let $\mu$ be as in Assumption 1 and consider as initial points the $j/n$-quantiles of $\mu$ with $j=1,\dots,n$. In case that $x^*$ is one of these quantiles, we modify this initial point by a displacement of order at least $n^{-\frac{1}{\kappa+1}}$. }  We emphasize that all of our assumptions 
are directly related to the initial configuration ($t=0$) of the particles, and do not require computing the evolved configuration of the particles around the critical time.

\medskip

Let us now proceed towards {our main results which differ depending on whether 
{$G_{\mu}''(x^*)$}
is zero or non-zero{, see \eqref{integral} and \eqref{def:Stieltjes}}. We first consider the case $G_{\mu}''(x^*)>0$, the case $G_{\mu}''(x^*)<0$ leading to analogous results with obvious changes.} To this end, let us define
\begin{equation}\label{def:tn}c_2:=\frac{2^{1/3}}{G_\mu''(x^*)^{1/3}t_{\rm cr}},\qquad
t_{n}^{\rm Ai}(\t):=t_{\rm cr}+\frac{2\tau}{c_2^2  n^{1/3}},
\end{equation}
for $\t\in\R$.
The parameter $\t$ should be seen as a new time parameter on a time scale of order $n^{-1/3}$ around the critical time $t_{\rm cr}$. Finally, in order to obtain a decent convergence of the kernel instead of convergence of the correlation functions \eqref{determinantality} merely, we define a conjugation of the kernel $K_{n,s,t}$ by
{\begin{align}
\tilde K_{n,s,t}(x,y):=K_{n,s,t}(x,y)\exp({f_n(t,y)-f_n(s,x)})\label{conjugation}
\end{align}}
with the gauge factors 
\begin{align}\label{gauge_factor}
f_n(s,x):=-nG_{\mu}(x^*)x+\frac{nG_{\mu}(x^*)^2s}2.
\end{align}
We note that any kernel of the form \eqref{conjugation} (with reasonable $f_n$) gives the same correlation functions as $K_{n,s,t}$. 
Moreover, we set
\begin{align}
x^{\Ai}_n(\t):=x^{*}+t_n^\Ai(\t)G_{\mu}(x^*).\label{def:xnAiry}
\end{align}
For $\t\leq 0$, we have
\begin{align}
x^{\Ai}_n(\t)=x^{*}(t_n^\Ai(\t)),
\end{align}
as is easily seen from {\eqref{def:Stieltjes}.} For $\t>0$, it is a linearization of the evolution $x^{*}(t_n^\Ai(\t))$, which is non-linear then.
 Our first main result now reads as follows.

\begin{thm}\label{thm:main}
Suppose that Assumptions 1--3 hold and that $G_{\mu}''(x^*)>0$. Then, as $n\to\infty$, we have for any $\s>0$ {and any $0<\e<\min\lb\frac{\k-2}{6(\k+1)},\frac1{15}\rb$}
\begin{align}
&\frac{1}{c_2n^{2/3}}\tilde K_{n,t_{n}^{\rm Ai}(\t_1),t_{n}^{\rm Ai}(\t_2)}\left(x^{\Ai}_n(\t_1)+\frac{u}{c_2n^{2/3}}, x^{\Ai}_n(\t_2)+\frac{v}{c_2n^{2/3}}\right)\\
&=\mathbb K^{\rm Ai}_{\t_1,\t_2}(u,v)+\O\lb \frac{e^{- \s(u+v)}}{n^{\e}}\rb,\label{thrm_1_convergence}
\end{align}
 where $\mathbb K^{\rm Ai}_{\t_1,\t_2}(u,v)$ is the extended Airy kernel given by
 \begin{align}\label{extended_Airy}
 \mathbb K^{\rm Ai}_{\t_1,\t_2}(u,v)&:=\frac{1}{(2\pi i)^2} \int_{\Sigma^{\rm Ai}} d\zeta \int_{\Gamma^{\rm Ai}} d\w ~ \frac{\exp\lb\frac{\zeta^3}3-\zeta v-\t_2\zeta^2-\frac{\w^3}3+\w u+\t_1\w^2\rb}{\zeta-\w}\\
&-1(\t_1>\t_2)\frac{1}{\sqrt{4\pi(\t_1-\t_2)}}\exp\lb-\frac{(u-v)^2}{4(\t_1-\t_2)}\rb.
 \end{align}
Here $\Sigma^{\Ai}$ consists of the two rays from $\infty e^{-i\frac{\pi}3}$ to 0 and from 0 to  $\infty e^{i\frac{\pi}3}$ and $\Gamma^{\rm Ai}$ consists of the two rays from $\infty e^{-i\frac{2\pi}3}$ to 0 and from 0 to  $\infty e^{i\frac{2\pi}3}$.

The $\O$-term in \eqref{thrm_1_convergence} is uniform with respect to $\tau_1,\t_2$ belonging to any compact subset of $\mathbb R$, and uniform with respect to $u,v\in [-M, Mn^{\epsilon}]$, where $M>0$ can be any positive constant. 
\end{thm}
\begin{figure}[h]
\centering
      \begin{minipage}[t]{0.45\linewidth}
			\centering
			 \includegraphics[trim=4cm  5cm 4cm 9cm,clip,width=\linewidth]{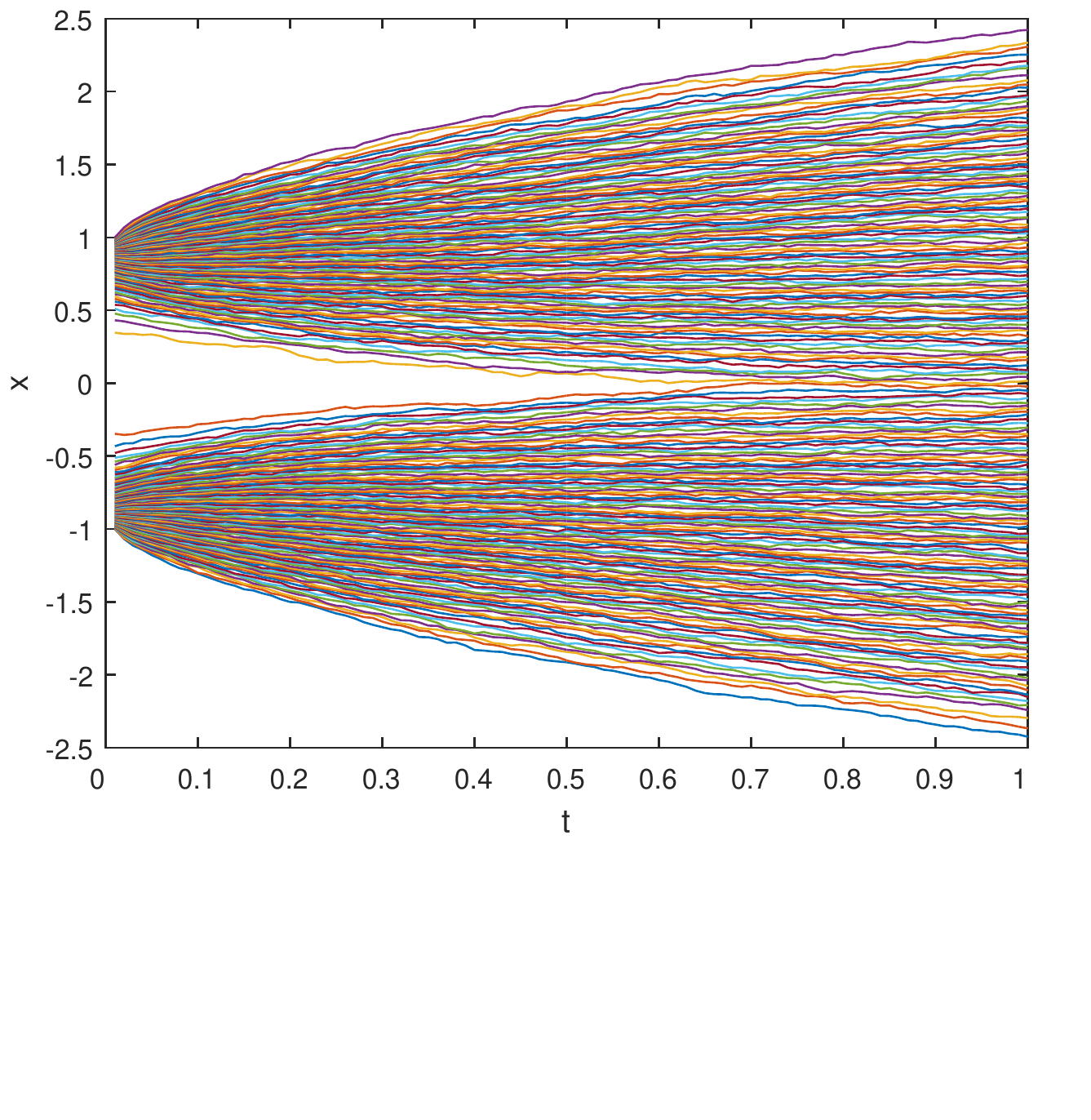}
			\end{minipage}%
			\hfill
			\begin{minipage}[t]{0.45\linewidth}
			\centering
\includegraphics[trim=4cm  5cm 4cm 9cm,clip,width=\linewidth]{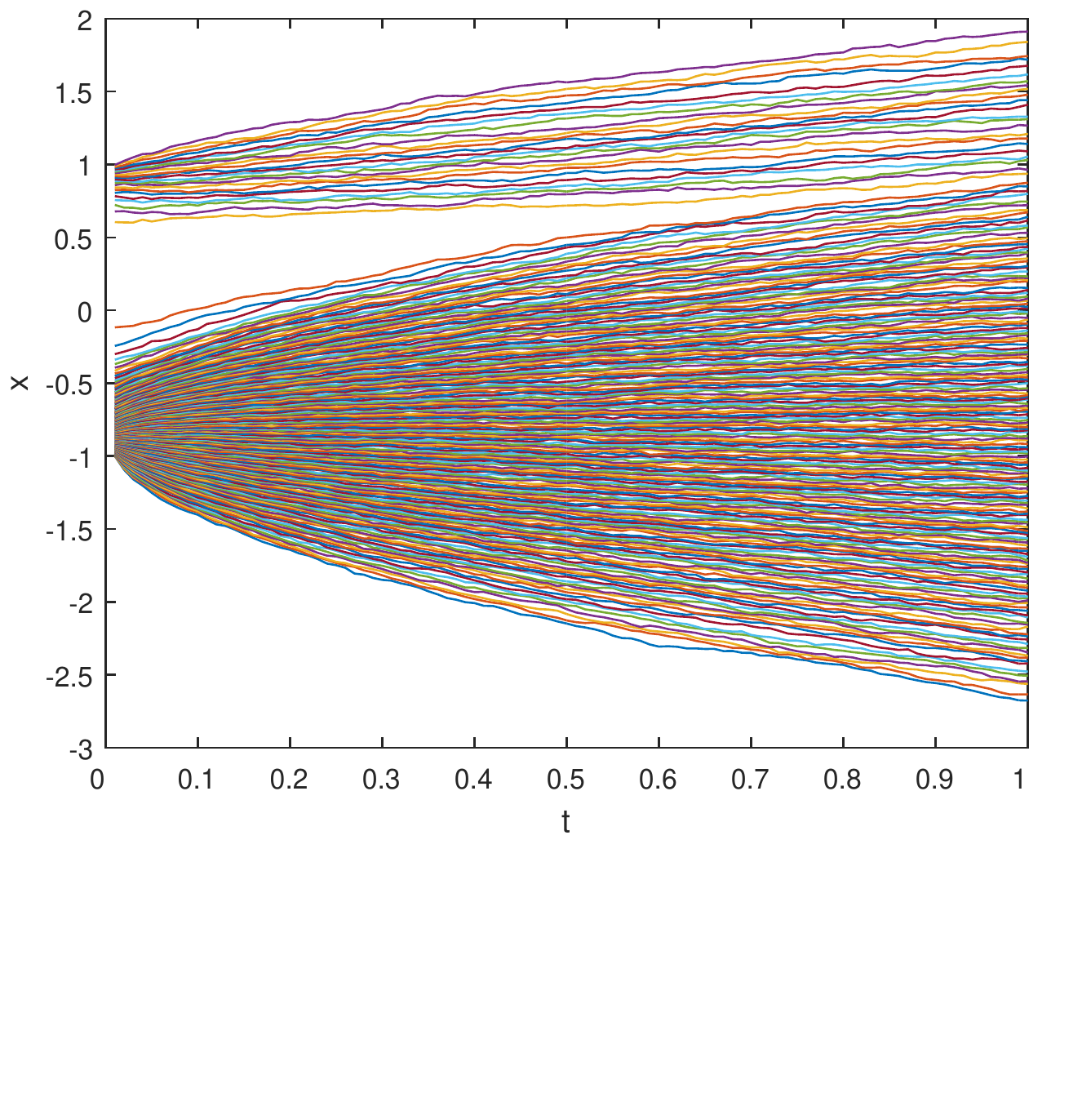} 
			\end{minipage} \vspace{-5em}\caption{Two samples of NIBM for $t\in[0,1]$ with $n=200$. The initial configuration consists of quantiles of the limiting measure $\mu$, with $d\mu(x)=\frac{5}{2}x^4 dx$ on $[-1,1]$ (left), and $d\mu(x)=c(x-0.2)^4 dx$ on $[-1,1]$ (right).
At the left, we have Pearcey statistics around the critical time $0.6$ at the critical point 0. At the right, we have Airy statistics around the critical time $\approx 0.7543$ at the critical point $\approx 0.7571$.
				}
			\label{figure1}
\end{figure}

\begin{rmk}\label{remark:Airy}\leavevmode
	\begin{enumerate}
		\item It is important to note that the exact definition of the extended Airy kernel is not unique in the literature. Essentially the same form as used here is given in \cite{DuseMetcalfe}. In \cite{Johansson03,BorodinKuan,Petrov}, the extended Airy kernel is defined as
		\begin{align*}
		\tilde{\mathbb  K}^{\rm Ai}_{\t_1,\t_2}(u,v):=\begin{cases}
		\int_0^\infty e^{-r(\t_2-\t_1)}\Ai(u+r)\Ai(v+r)dr,&\quad\text{if }\ \t_2\geq\t_1\\
		-\int_{-\infty}^0 e^{-r(\t_2-\t_1)}\Ai(u+r)\Ai(v+r)dr,&\quad\text{if }\ \t_2<\t_1.
		\end{cases}
		\end{align*}
		Represented as a double contour integral, $\tilde{\mathbb  K}^{\rm Ai}_{\t_1,\t_2}$ reads \cite[Equation (8.7)]{Petrov}
		\begin{align*}
		&\tilde{\mathbb K}^{\rm Ai}_{\t_1,\t_2}(u,v)\\
		&=\frac{1}{(2\pi i)^2} \int\limits_{\Sigma^{\rm Ai}} d\zeta \int\limits_{\Gamma^{\rm Ai}} d\omega~\frac{\exp\lb \frac{(\zeta-\t_2)^3}3-v(\zeta-\t_2)-\frac{(\w-\t_1)^3}3+u(\w-\t_1) \rb}{\zeta-\omega}\\
		&\quad -1(\t_1>\t_2)\frac{1}{\sqrt{4\pi(\t_1-\t_2)}}\exp\lb-\frac{(u-v)^2}{4(\t_1-\t_2)}-\frac12(\t_1-\t_2)(u+v)+\frac1{12}(\t_1-\t_2)^3\rb
		\end{align*}
		with the same contours as in \eqref{extended_Airy}. It is then straightforward to check that 
		{\begin{align}
		\exp\lb-\t_1u+\t_2v+\frac{\t_1^3-\t_2^3}3\rb\mathbb K^{\rm Ai}_{\t_1,\t_2}(u-\t_1^2,v-\t_2^2)=\tilde{\mathbb K}^{\rm Ai}_{\t_1,\t_2}(u,v).
		\end{align}} 
		Thus  $\mathbb K^{\rm Ai}_{\t_1,\t_2}$ and $\tilde{\mathbb K}^{\rm Ai}_{\t_1,\t_2}$ are equivalent in the sense that one can be transferred to the other via a conjugation and a change of variables. 
{		\item In \eqref{thrm_1_convergence}, the points $x^{\Ai}_n(\t_1)$ and $x^{\Ai}_n(\t_2)$ can not be replaced by the common $x^*(t_{\rm cr})=x^{\Ai}_n(0)$ if $\t_1\not=\t_2$. This shows how delicate the convergence in \eqref{thrm_1_convergence} is.
		}
		\item The extended Airy kernel generates a time-dependent determinantal point process called Airy line ensemble (or sometimes multi-line Airy process) \cite{CorwinHammond}. It can be seen as a collection of infinitely many paths that may be assumed continuous. It is stationary, a Markov process \cite{KatoriTanemura11} and the distribution of each path is locally absolutely continuous w.r.t.~the distribution of a multiple of Brownian motion \cite{CorwinHammond}.
		\item We do not consider the case of vanishing orders $1<\k\leq2$ in this work as it seems to be genuinely different from the case $\k>2$. In {the situation $1<\k\leq2$,} the integral \eqref{integral} does not exist, and it can be expected that for these $\k$, the exact form of $\mu_n$ and its limit will play a more prominent role for the asymptotic behavior. A similar remark applies to the Pearcey case below, in which we assume $\k>3$.
	\end{enumerate}
\end{rmk}
An important feature of the Airy line ensemble is the almost sure existence of a finite largest path, indicating that the Airy line ensemble occurs when a one-sided gap is present. {This largest path is called Airy$_2$ process and is well-known to occur at spectral edges of NIBM, generally describing the fluctuations of the largest eigenvalue of a bulk of eigenvalues.} Our finding of Airy correlations at the critical point is quite remarkable as no {gap} in the spectrum can be seen on a global scale (see e.g.~Figure \ref{Criticalden}). This raises the question of a gap at a smaller scale. We will address this question in our next theorem. We show that there is at any time $t_n^\Ai(\t)$, with high probability, indeed a gap in the random spectrum to the right of $x_n^\Ai(\t)$ at a mesoscopic scale, for whose size we give a lower bound. Moreover, we {identify} a particle around {$x^\Ai_n(\t)$} whose fluctuations follow the Airy$_2$ process. The Airy$_2$-process $(\A(\t))_{\t\in\R}$ has been introduced in \cite{PS} and may be defined by its finite-dimensional distributions as Fredholm determinants,
\begin{align}
P(\A(\t_1)\leq a_1,\dots,\A(\t_m)\leq a_m):=\det(I-\mathcal K^{\rm Ai}_{a_1,\dots,a_m} )_{L^2(\{\t_1,\dots,\t_m\}\times\R,\#\otimes \lambda)},\label{Airy_2_def}
\end{align}
where $\#$ and $\lambda$ denote counting and Lebesgue measure, respectively, {and $\mathcal K^{\rm Ai}_{a_1,\dots,a_m}$ is the integral operator acting on $L^2(\{\t_1,\dots,\t_m\}\times\R,\#\otimes \lambda)$ defined by}
\begin{align}\label{def_Fredholm_limit}
(\mathcal K^{\rm Ai}_{a_1,\dots,a_m}g)(\t,u):=\int_{\{\t_1,\dots,\t_m\}\times\R}r(\t,u)\mathbb K^{\rm Ai}_{\t,\t'}(u,v)r(\t',v)g(\t',v)\#(d\t')dv,
\end{align} 
and $r(\t_j,x):=1_{(a_j,\infty)}(x), \,j=1,\dots,m$.  

Alternatively, we may write the determinant in \eqref{Airy_2_def} as
\begin{align}
\det(I-\mathcal K^{\rm Ai})_{\bigoplus_{j=1}^m L^2((a_j,\infty),d\lambda)}
\end{align}
with $\mathcal K^{\rm Ai}$ being the integral operator on $\bigoplus_{j=1}^m L^2((a_j,\infty),d\lambda)$ with (block) kernel $(\mathbb K^{\rm Ai}_{\t_i,\t_j})_{1\leq i,j\leq m}$. A more explicit representation of Fredholm determinants will be given in the proof of Theorem \ref{corollary_TW} below.

The Airy$_2$-process is stationary and has a continuous version (see \cite{Johansson}). At any time $\t$, the distribution of $\mathcal A(\t)$ is the Tracy-Widom distribution (with parameter $\b=2$).

{\begin{thm}\label{corollary_TW}
	Let $0<\e<\min\lb\frac{\k-2}{6(\k+1)},\frac1{15}\rb$, $m \in \mathbb N$, $\t_1<\dots<\t_m\in\R$.
	\begin{enumerate}
		\item Let $0<\e'<\e$. Then there is a $\d>0$ such that for $n$ large enough,
		\begin{align}
		&\P\lb X_i(t_n^\Ai(\t_j))\notin\left[ x_n^\Ai(\t_j)+n^{\e'-\frac23},x_n^\Ai(\t_j)+n^{\e-\frac23}\right],\ i=1,\dots,n,\,j=1,\dots,m\rb\\
		&\geq 1-e^{-n^{\d}}.
		\end{align}
		\item	Define $\xi(\tau)$ as the largest particle of $X(t_{n}^{\rm Ai}(\t))$ which lies below the threshold value $x_n^{\rm Ai}(\t)+n^{-2/3+\e}$. Then we have for any $a_1,\dots,a_m\in\R$ as  $n\to\infty$
	\begin{align}
	&\P\left( c_2n^{2/3}(\xi(\t_1)-x_n^{\Ai}(\t_1))\leq a_1,\dots,c_2n^{2/3}(\xi(\t_m)-x_n^{\Ai}(\t_m))\leq a_m\right)\\
	&= \det(I-\mathcal K^{\rm Ai}_{a_1,\dots,a_m} )_{L^2(\{\t_1,\dots,\t_m\}\times\R,\#\otimes \lambda)}+\O(n^{-\e})
	\end{align}
	with an $\O$ term that is uniform for $a_1,\dots,a_m$ in compacts. In particular, as $n\to\infty$,
	\begin{align}
	(c_2n^{2/3}(\xi(\t)-x_n^{\rm Ai}(\t)))_{\t\in\R}\to (\A(\t))_{\t\in\R}\label{conv:xi}
	\end{align}
	in the sense of weak convergence of finite-dimensional distributions.
\end{enumerate}
\end{thm}}
\begin{rmk}
\begin{enumerate}
	\item In particular, the distribution of $c_2n^{2/3}(\xi(\t)-x^{\Ai}_n(\t))$ converges to the Tracy-Widom distribution for any $\t\in\R$.
	\item We will not consider the question of convergence of \eqref{conv:xi} as a random function, say locally uniformly in $\t$. Instead, we refer the reader interested in this question to a number of recent articles dealing with the required tightness property of NIBM \cite{CorwinHammond,Virag1,Virag2}.
	\item Theorem \ref{corollary_TW} shows that at any finite time there is {with high probability} a mesoscopic gap of size $n^{\e-2/3}$ in the unnormalized spectrum. To our knowledge, this is the first instance of an {occurrence} of Airy fluctuations and Tracy-Widom statistics at a mesoscopic {rather than} a macroscopic gap. The bound on the size of this gap is not optimal, but chosen for convenience in the proofs of Theorems \ref{thm:main} and \ref{corollary_TW}. 
\end{enumerate}
\end{rmk}
\vspace{2em}

{Let us now consider the case $G_{\mu}''(x^*)=0$. P}rovided that $\kappa>3$, we observe limiting correlations given by the extended Pearcey kernel introduced in \cite{TracyWidom2}, which is
\begin{align*}
\mathbb K^{\rm P}_{\t_1,\t_2}(u,v):&=\frac{1}{(2\pi i)^2} \int_{i\R} d\zeta \int_{\Gamma^{\rm P}} dw ~ \frac{\exp\lb-\frac{\zeta^4}4-\frac{\t_2\zeta^2}2-v\zeta +\frac{\w^4}4+\frac{\t_1\w^2}2+u\w\rb}{\zeta-\w}\\
&-1(\t_1>\t_2)\frac{1}{\sqrt{2\pi(\t_1-\t_2)}}\exp\lb-\frac{(u-v)^2}{2(\t_1-\t_2)}\rb.
\end{align*}
Here $\Gamma^{\rm P}$ consists of four rays, two from the origin to $\pm\infty e^{-i\pi/4}$ and two from $\pm\infty e^{i\pi/4}$ to the origin. Let us parametrize 
\begin{align}\label{def:tnP}
t_{n}^{\rm P}(\t):=t_{\rm cr}+\frac{\t}{c_3^2 n^{1/2}},\quad\t\in\R
\end{align}
and set (note that $G_{\mu}'''(x^*)<0$)
$$c_3:=\frac{1}{t_{\rm cr}}\left(\frac{6}{-G_{\mu}'''(x^*)}\right)^{1/4},\quad x^{\rm P}_n(\t):=x^*+t_n^{\rm P}(\t)G_{\mu}(x^*).$$

\begin{thm}\label{thm:main_Pearcey}
	Suppose that Assumptions 1--3 hold and that $G_{\mu}''(x^*)=0$ and $\k>3$. Then, as $n\to\infty$, we have {for any $0<\e<\min\lb\frac{\k-3}{8(\k+1)},\frac1{24}\rb$ }
	\[
	\frac{1}{c_3n^{3/4}}\tilde K_{n,t_{n}^{\rm P}(\tau_1),t_{n}^{\rm P}(\t_2)}\left(x^{\rm P}_n(\t_1)+\frac{u}{c_3n^{3/4}}, x^{\rm P}_n(\t_2)+\frac{v}{c_3n^{3/4}}\right)=\mathbb K^{\rm P}_{\t_1,\t_2}(u,v)+\O(n^{-\e}).\]
	The convergence is uniform for $\tau_1,\t_2$ in any compact subset of $\mathbb R$, and uniform for $u,v\in [-M, M]$, where $M>0$ can be any positive constant.
\end{thm}

{ The extended Pearcey kernel  {is supposed to generate} a space-time determinantal point process called Pearcey process{, which does generically not have a smallest or largest path.}
%		{\item	Since the Hermitian Brownian motion $(M(t))_{t\geq0}$ is a centered Gaussian process, we can write the density of $(M(t_1),\dots,M(t_k))$ on $\lb\C^{n\times n}\rb^k$ as being proportional to
%	\begin{align*}
%	e^{-\sum_{j=1}^k\frac{1}{2(t_j-t_{j-1})}\Tr (M(t_j)-M(t_{j-1}))^2},
%	\end{align*}
%	where $t_0:=0$. Thus we can also understand space-time correlations of the process $(X(t))_{t\geq0}$ for finite times as spatial eigenvalue correlations of a multi-matrix model (cf.~\cite{EynardMehta}).}

Let us conclude this introduction with a discussion of our results in the context of the existing literature. Briefly speaking, the occurrence of the Airy line ensemble for the largest eigenvalues of NIBM {(and of its variants  like Ornstein-Uhlenbeck processes \cite{Dyson} or Brownian bridges)} has been common knowledge for quite some time.  To our knowledge, convergence of NIBM to the extended Airy kernel  was studied first in \cite{FNH} for a single time. Convergence of the finite-dimensional distributions (multi-time) of the top path was stated in \cite{Johansson03} without a proof, the process being started in $n$ random GUE points. For half of the particles starting at $-a$ and half at $a$, single time correlations have been studied in \cite{BK1,ABK2,BK3}, Airy correlations {at the edges being proved in \cite{BK1,ABK2} for $a<1$ and $a>1$, respectively, and Pearcey fluctuations at the cusp in \cite{BK3} for $a=1$}. A proof of the convergence in law to the multi-line version of the Airy$_2$ process, there coined Airy line ensemble, has been  {given} in \cite{CorwinHammond} for the initial condition of all particles starting at 0. For a large class of initial deterministic or random conditions, the Airy point process was found in \cite{Shcherbinaedge} for the single-time fluctuations around the outmost edges of the limit $\mu$. {A study for single-time correlations (via the deformed GUE \eqref{deformed_GUE}) has been provided in \cite{CapitainePeche}; it has been shown in \cite{CapitainePeche} that Airy point process statistics appear whenever the evolved density $\psi_t$ has an edge with a square root decay on one side and a macroscopic gap on the other side{, a result which was extended in \cite{AEKS} to deformed Wigner matrices}. Recall that in the present paper, we observe the multi-time correlations of the Airy line ensemble around a mesoscopic instead of a macroscopic gap, and we obtain them under assumptions on the initial density rather then the evolved density.}

Statistics of the Airy line ensemble are universal for a large class of random matrix ensembles and related models \cite{Soshnikov,PasturShcherbina03,BEY,EYbook,KSSV,KV} as well as many random growth models and discrete interacting particle systems belonging to the KPZ universality class{, see e.g. \cite{Corwinsurvey,SpohnLN,KK} and references therein.} {Let us also mention studies of discrete analogs of NIBM, see e.g.~\cite{BorodinKuan,Johansson08,Petrov,DJM,DuseMetcalfe,GorinPetrov} and references therein.}\\

Pearcey statistics have also been  {found} in numerous publications in random matrix theory. Let us mention \cite{BrezinHikami,TracyWidom2,BK3,LiechtyWang1,LiechtyWang2,GeudensZhang,CapitainePeche,Erdosetal}. Our results on the Pearcey correlations should in particular be compared to those of \cite{CapitainePeche} {and \cite{Erdosetal}}. In \cite{CapitainePeche}, single-time Pearcey kernel statistics are shown to occur if the evolved density $\psi_t$ has an isolated zero at some point $x^*$ and time $t=1$ and moreover the following condition holds: there is a complex neighborhood $B$ of $x^*$ such that the equation 
\begin{align}
\int\frac{d\mu_n(s)}{(z-s)^2}=1 \label{condition_CP}
\end{align} 
has a unique solution $z\in B$ for $n$ sufficiently large. As $\mu_n$ is discrete, this condition may be written as a polynomial equation $P(z)=0$ with a polynomial $P$ with real coefficients of degree $2n$. 
Since the non-real solutions of this equation come in complex conjugate pairs, the unique zero $z\in B$ of \eqref{condition_CP} must be real. Moreover, if we assume for simplicity that the support of $\mu_n$ is contained in the support of its weak limit $\mu$ as $n\to\infty$, then $\int\frac{d\mu_n(s)}{(z-s)^2}$ converges to $\int\frac{d\mu(s)}{(z-s)^2}$ for $z$ outside any complex compact set containing a neighborhood of the support of $\mu$. Hence, all the solutions of \eqref{condition_CP} will converge either to one of the finitely many solutions of the equation $\int\frac{d\mu(s)}{(z-s)^2}=1$ outside the support of $\mu$, or to the support of $\mu$. As $n\to\infty$, we thus have a growing number of solutions to \eqref{condition_CP} converging to the support of $\mu$.
This discussion makes us believe that the condition of having a fixed neighborhood $B$ of $x^*$ where there is only (and exactly) one solution for large $n$ is rather restrictive, and although this is not so easy to verify even in simple concrete situations, it makes us believe that the condition of \cite{CapitainePeche} is typically not satisfied under our assumptions.

{Two other deep studies of Pearcey universality are \cite{Erdosetal} and \cite{CEKS}. In \cite{Erdosetal}, the authors investigate deformed Hermitian Wigner matrices}, including the deformed GUE. {In \cite{CEKS}, deformed real-symmetric Wigner matrices are treated, including the deformed Gaussian orthogonal ensemble (GOE), for which also space-time correlations are considered.}   The setting of \cite{Erdosetal} is quite general but specialized to the deformed GUE, relevant for Pearcey universality is, apart from the evolved density showing a cusp-like behavior, an abstract condition on boundedness of a vector-valued solution to a certain system of non-linear equations {(called Dyson equations, cf.~\cite{AEK,AEK2})}, uniform in a complex neighborhood of the critical point. This vector-valued function $m=(m_1,\dots,m_n)$ is closely related to the Stieltjes transform of the evolved density.
For a concrete quantitative condition for the uniform boundedness of $m_i$, $i=1,\dots,n$, \cite{Erdosetal} refers to \cite{AEK}. Adapted to our setting, \cite[Theorem 6.4, Lemma 6.7]{AEK} provides the required boundedness (and thus Pearcey universality) under the condition that the function $Q:[0,\infty)\to[0,\infty)$,
	\begin{align}
	Q(q)^2:=\min_{1\leq i\leq n}\frac1n\sum_{j=1}^n\frac1{(q^{-1}+\lv X_i(0)-X_j(0)\rv)^2}\label{def_Q}
	\end{align}
	satisfies $\lim_{q\to\infty}Q(q)>\max(5/t_{\rm cr},1)$ for $n$ large enough. This is however not the case in our situation, as the r.h.s.~of \eqref{def_Q} should for $X_i(0)$ close to $x^*$ and $n,q\to\infty$ converge to $\int(x^*-s)^{-2}d\mu(s)=1/t_{\rm cr}$.
	The deeper reason for the non-applicability of the condition from \cite{AEK} is that this condition works well if the $X_j(0)$'s have neighbors at distances $\O(n^{-1/2})$, but in our case, there are gaps in the initial spectrum of order  $n^{-\frac{1}{\k+1}}$ and $\k>2$.\\

{The remainder of the paper is organized as follows. We will prove Theorems \ref{thm:main} and \ref{thm:main_Pearcey} using a saddle point analyis of the kernel $\tilde{K}_{n,s,t}$. While this approach to the asymptotics of NIBM goes back already to \cite{Johansson01}, it has not been applied in the situation of a {strongly vanishing initial density}.
	An important new ingredient to this analysis is an expansion of the Stieltjes transform $G_{\mu_n}$ on the real line, uniformly in a mesoscopic neighborhood of the initial critical point $x^*$. This expansion is obtained in Section 2.}  Section 3 contains the {saddle point analysis} of the Airy case, including the proofs of Theorems \ref{thm:main} and \ref{corollary_TW}. The Pearcey case will then be studied in Section 4, where in particular Theorem \ref{thm:main_Pearcey} is proven. {Finally, definition and determinantal representation of the space-time correlation functions are given in Appendix A.}\\

\textbf{Acknowledgements:} The authors would like to thank Mireille Capitaine, {Giorgio Cipolloni,} Makoto Katori, Torben Kr\"uger, Sandrine P\'ech\'e, {Dominik Schr\"oder} {and Dong Wang} for valuable discussions. T.C. was supported by the
Fonds de la Recherche Scientifique-FNRS under EOS project O013018F. T.N.~and M.V.~have been supported by the DFG through the CRC 1283 ``Taming uncertainty and profiting from
randomness and low regularity in analysis, stochastics and their applications''.

\section{{Expansions of the Stieltjes Transform at the Critical Point}}
{The asymptotic analysis of \eqref{def:Kn} needs precise information about the large $n$ behavior of the function $g_{\mu_n}$ around the point $x^*$. In this preparatory section, we provide an expansion for the Stieltjes transform $G_{\mu_n}$ (i.e.\ the derivative of $g_{\mu_n}$, see \eqref{def:Stieltjes})  around the critical point $x^*$.

{To state it, define
\begin{equation}\label{def:G0G1G2}G_j:=G_\mu^{(j)}(x^*)=(-1)^{j}j!\int\frac{d\mu(s)}{(x^*-s)^{j+1}},\quad\mbox{for $j=0,1,2,3$.}\end{equation}}
\begin{prop}\label{prop: Stieltjes comparison 2}\leavevmode
	\begin{enumerate}
	\item 
	Suppose that Assumptions 1--3 hold, and that $\kappa>2$. Then we have {for any $0<\e<\min\lb\frac{\k-2}{6(\k+1)},\frac1{15}\rb$} uniformly with respect to $|z-x^*|=\mathcal{O}\left(n^{-\frac{1}{3}+\epsilon}\right)$ 
	\begin{equation} \label{eq:expansionG_mun2}
	G_{\mu_n}(z)=G_0+G_1(z-x^*)+\frac{G_2}{2}(z-x^*)^2+\O(n^{-2/3-2\e}),
	\end{equation} as $n\to\infty$.
	Moreover, we have uniformly for $|z-x^*| =\mathcal{O}\left(n^{-\frac{1}{3}+\epsilon}\right)$
	\begin{equation}\label{eq:lemma extension2}
	G_{\mu_n}''(z)=G_2+\O(n^{-\e}),
	\end{equation}
	as $n\to\infty$.
	\item Suppose that Assumptions 1--3 hold, and that $\kappa>3$. Then we have {for any $0<\e<\min\lb\frac{\k-3}{8(\k+1)},\frac1{24}\rb$} uniformly with respect to $|z-x^*|=\mathcal{O}\left(n^{-\frac{1}{4}+\epsilon}\right)$, 
	\begin{equation} \label{eq:expansionG_mun}
	G_{\mu_n}(z)=G_0+G_1(z-x^*)+\frac{G_2}{2}(z-x^*)^2+\frac{G_3}{6}(z-x^*)^3+\O(n^{-3/4-2\e}),
	\end{equation} as $n\to\infty$.	
	Moreover, we have uniformly for $|z-x^*| =\mathcal{O}\left(n^{-\frac{1}{4}+\epsilon}\right)$
	\begin{equation}\label{eq:lemma extension}
	G_{\mu_n}''(z)=G_2+\O(n^{-\e}),\qquad G_{\mu_n}'''(z)=G_3+\O(n^{-\e}),
	\end{equation}
	as $n\to\infty$.
	\end{enumerate}
\end{prop}
\begin{proof}We will first prove the second part of the proposition in detail, afterwards we indicate briefly how the first part can be proved similarly.

 Let
	\begin{align}
	0<\e<\min\lb\frac{\k-3}{8(\k+1)},\frac1{24}\rb,
	\end{align}
	a choice that will become clear in the course of the proof. We first write $G_{\mu_n}(z)$ as
	\[
	G_{\mu_n}(z)=\int\frac{1}{z-s}d\mu_n(s)
	=-\int_{\mathbb R}\frac{F_n(s)}{(z-s)^2}ds,
	\]
	where $F_n$ is as before the distribution function of $\mu_n$.
	Since $0<\e<\frac{\k-3}{8(\k+1)}$, we have with Assumption 3 that $I_n:=[x^*-n^{-\frac{1}{4}+2\epsilon}, x^*+n^{-\frac{1}{4}+2\epsilon}]$ does not intersect with the support of $\mu_n$ for $n$ large enough, hence \(F_n\) is constant on $I_n$ and we can write
	\begin{align*}
	G_{\mu_n}(z)
	&=-\int_{\mathbb R\setminus I_n}\frac{F_n(s)}{(z-s)^2}ds +\frac{F_n(x^*-n^{-\frac{1}{4}+2\epsilon})}{z-x^*+n^{-\frac{1}{4}+2\epsilon}}
	-\frac{F_n(x^*+n^{-\frac{1}{4}+2\epsilon})}{z-x^*-n^{-\frac{1}{4}+2\epsilon}}.
	\end{align*}

	Next, let $\widetilde\mu$ be the restriction of the measure $\mu$ to $\mathbb R\setminus I_n$, i.e.~the measure with density $\psi\cdot1_{\R\setminus I_n}$.
	Then we have, $F$ being the distribution function of $\mu$,
	\begin{align*}
	G_{\widetilde\mu}(z)&=\int_{-\infty}^{x^*-n^{-\frac{1}{4}+2\epsilon}}\frac{d\mu(s)}{z-s}
	+\int_{x^*+n^{-\frac{1}{4}+2\epsilon}}^{+\infty}\frac{d\mu(s)}{z-s}\\
	&=-\int_{-\infty}^{x^*-n^{-\frac{1}{4}+2\epsilon}}\frac{F(s)ds}{(z-s)^2}
	-\int_{x^*+n^{-\frac{1}{4}+2\epsilon}}^{+\infty}\frac{F(s)ds}{(z-s)^2}\\&\qquad +\frac{F(x^*-n^{-\frac{1}{4}+2\epsilon})}{z-x^*+n^{-\frac{1}{4}+2\epsilon}}
	-\frac{F(x^*+n^{-\frac{1}{4}+2\epsilon})}{z-x^*-n^{-\frac{1}{4}+2\epsilon}}.
	\end{align*}
	Using Assumption 2, it follows that uniformly for $|z-x^*|\leq Cn^{-\frac{1}{4}+\epsilon}$ for any fixed $C>0$
	\begin{align*}\left|G_{\mu_n}(z)-G_{\widetilde\mu}(z)\right|\leq &\frac{M}{n}\int_{\mathbb R\setminus I_n}\frac{ds}{\vert z-s \vert^2}\\
	&+\frac{M}{n}\frac{1}{\vert z-x^{*}-n^{-\frac{1}{4}+2\epsilon}\vert}+\frac{M}{n}\frac{1}{\vert z-x^{*}+n^{-\frac{1}{4}+2\epsilon}\vert}\\
	&=\mathcal O(n^{-\frac{3}{4}-2\epsilon}),
	\end{align*}
	as $n\to\infty$, $M$ being the constant from Assumption 2.

	Thus, it remains to show, with the same uniformity in $z$, that 
	\begin{equation}\label{eq:expansionG}
	G_{\widetilde\mu}(z)=G_0+G_1(z-x^*)+\frac{G_2}{2}(z-x^*)^2+\frac{G_3}{6}(z-x^*)^3+\mathcal O(n^{-\frac{3}{4}-2\epsilon}),
	\end{equation}
	as $n\to\infty$.
	As $G_{\widetilde\mu}$ is analytic in $|z-x^*|\leq Cn^{-\frac{1}{4}+\epsilon}$ for $n$ large enough, we have
	\begin{align}\label{G_mu_tilde}
	G_{\widetilde\mu}(z)=&G_{\widetilde\mu}(x^*)+G_{\widetilde\mu}'(x^*)(z-x^*)+\frac{1}{2}G_{\widetilde\mu}''(x^*)(z-x^*)^2\\
	&+\int_{x^*}^{z}\int_{x^*}^\xi \int_{x^*}^\zeta G_{\widetilde\mu}'''(\eta)d\eta d\zeta d\xi.
	\end{align}
	Moreover, using the definition of $\widetilde\mu$, it is straightforward to see that 
	\begin{align*}
	&G_{\widetilde\mu}(x^*)=G_0+\mathcal O(n^{-\frac{\kappa}{4}+2\kappa\epsilon}),\\
	&G_{\widetilde\mu}'(x^*)=G_1+\mathcal O(n^{-\frac{\kappa-1}{4}+2(\kappa-1)\epsilon}),\\
	&G_{\widetilde\mu}''(x^*)=G_2+\mathcal O(n^{-\frac{\kappa-2}{4}+2(\kappa-2)\epsilon}),
\end{align*} 
	as $n\to\infty$. To deal with the remainder in \eqref{G_mu_tilde}, we write
\begin{align}
G_{\widetilde\mu}'''(\eta)-G_3=-6\int\limits_{\mathbb{R}\backslash I_n} \left(\frac{1}{(\eta -s)^4}-\frac{1}{(x^* -s)^4}\right) d\mu(s)+6 \int\limits_{I_n} \frac{1}{(x^* -s)^4} d\mu(s).\label{G_mu_remainder}
\end{align}
The last integral in \eqref{G_mu_remainder} is of order \(\mathcal O \left(n^{(\kappa -3) (-\frac{1}{4} +\epsilon)}\right)\) and for the first integral it follows from a computation that we have
\[\int\limits_{\mathbb{R}\backslash I_n} \left(\frac{1}{(\eta -s)^4}-\frac{1}{(x^* -s)^4}\right) d\mu(s) = \begin{cases}
\mathcal O \left(n^{(\kappa -3) (-\frac{1}{4} +\epsilon)}\right), & 3<\kappa <4, \\
\mathcal O \left(n^{-\frac{1}{4} +\epsilon} \log n \right), & \kappa =4,\\
\mathcal O \left(n^{-\frac{1}{4} +\epsilon}\right), & \kappa >4,
\end{cases}\]
as \(n\to\infty\), uniformly in \(\vert \eta -x^* \vert \leq Cn^{-\frac{1}{4}+\epsilon }\).
	This yields, using $\e<\frac1{24}$,
	\begin{align}
		&\int_{x^*}^{z}\int_{x^*}^\xi \int_{x^*}^\zeta G_{\widetilde\mu}'''(\eta)d\eta d\zeta d\xi =\frac{G_3}{6}\left(z-x^*\right)^3+\mathcal O(n^{-\frac{3}{4}-2\epsilon}).
	\end{align} 
	 With this, \eqref{eq:expansionG} follows.

	In order to see \eqref{eq:lemma extension} for \(\vert z-x^{*}\vert \leq C n^{-\frac{1}{4}+\epsilon}\), we write
	\[G_{\mu_n}'''(z)-G_{\mu}'''(x^*)=G_{\mu_n}'''(z)-G_{\widetilde\mu}'''(z) +G_{\widetilde\mu}'''(z)-G_3.\]
	It follows from the above arguments that $G'''_{\tilde\mu}(z)-G_3=\O(n^{-\e})$,
	so it remains to show for \(\vert z-x^{*}\vert \leq C n^{-\frac{1}{4}+\epsilon}\)
	\[G_{\mu_n}'''(z)-G_{\widetilde\mu}'''(z) =\mathcal O \left(n^{-\e}\right),\quad n\to \infty.\]
	To see this, we write
	\[G_{\mu_n}'''(z) = 24\int\limits_{\mathbb{R}\backslash I_n} \frac{F_n(s)}{(z-s)^5}ds+6 \frac{F_n(x^{*}+n^{-\frac{1}{4}+2\epsilon})}{(z-x^{*}-n^{-\frac{1}{4}+2\epsilon})^4}-6 \frac{F_n(x^{*}-n^{-\frac{1}{4}+2\epsilon})}{(z-x^{*}+n^{-\frac{1}{4}+2\epsilon})^4},\]
	and
	\[G_{\widetilde\mu}'''(z) = 24\int\limits_{\mathbb{R}\backslash I_n} \frac{F(s)}{(z-s)^5}ds+6 \frac{F(x^{*}+n^{-\frac{1}{3}+2\epsilon})}{(z-x^{*}-n^{-\frac{1}{4}+2\epsilon})^4}-6 \frac{F(x^{*}-n^{-\frac{1}{4}+2\epsilon})}{(z-x^{*}+n^{-\frac{1}{4}+2\epsilon})^4}.\]
	Using now Assumption 2 as above, we find
	\begin{align*}\vert G_{\mu_n}'''(z)-G_{\widetilde\mu}'''(z)\vert &\leq \frac{24M}{n} \int\limits_{\mathbb{R}\backslash I_n} \frac{ds}{\vert z-s \vert^5}\\
	 &+\frac{6M}{n} \frac{1}{\vert z-x^{*} -n^{-\frac{1}{4}+2\epsilon} \vert^4}+\frac{6M}{n} \frac{1}{\vert z-x^{*} +n^{-\frac{1}{4}+2\epsilon} \vert^4}=\O(n^{-8\e}),
	\end{align*}
	 uniformly for \(\vert z-x^{*}\vert \leq C n^{-\frac{1}{4}+\epsilon}\). Similar arguments yield the estimate for the second derivative in \eqref{eq:lemma extension}. This proves part (2) of the proposition.
	
	 For part (1) with $\kappa>2$, we proceed in a similar way as before, but now we have only a smaller interval $I_n:=[x^*-n^{-\frac{1}{3}+2\epsilon}, x^*+n^{-\frac{1}{3}+2\epsilon}]$ which does not intersect with the support of $\mu_n$ for small $\epsilon$. With this modification, it suffices to go through the same estimates as before to prove the results, where now
	\begin{align}
	0<\e<\min\lb\frac{\k-2}{6(\k+1)},\frac1{15}\rb.
	\end{align}

\end{proof}

\section{The Airy Case: Proof of Theorems \ref{thm:main} and \ref{corollary_TW}}\label{sec_Airy}
{Throughout this section, $\e$ can be any number satisfying $0<\e<\min\lb\frac{\k-2}{6(\k+1)},\frac1{15}\rb$.}
\subsection{Analysis of the Saddle Points}
	In view of \eqref{def:Kn}, for the proof of Theorem \ref{thm:main} we want to determine the asymptotic location of the saddle points of the functions $z\mapsto\phi_{n,\t_2}(z,v)$ and $w\mapsto-\phi_{n,\t_1}(w,u)$, where
	\begin{equation}
	\phi_{n,\t}(z,v):= \frac{n}{2t_n^\Ai(\t)} \left(z-x_n^{\Ai}(\t)-\frac{v}{c_2 n^{2/3}}\right)^2 + ng_{\mu_n}(z),\label{def:phi}
	\end{equation}
	 {and} we recall that $t_n^\Ai(\t)$ was defined in \eqref{def:tn}.
	 In the following of this section, for convenience in notation we will drop the superscript $\Ai$, simply writing $t_n(\t)$ and $x_n(\t)$ instead of \(t_n^\Ai(\t)\) and \(x_n^{\Ai}(\t)\), respectively.
	The saddle points in the upper half plane are given by
	\begin{align}
	&z_n := z_n(\t_2,v):=F_{t_n(\t_2),\mu_n}\left(x_n(\t_2)+\frac{v}{c_2 n^{2/3}}\right),\label{def:znwn1}\\
	&w_n :=w_n(\t_1,u):= F_{t_n(\t_1),\mu_n}\left(x_n(\t_1)+\frac{u}{c_2 n^{2/3}}\right), \label{def:znwn}
	\end{align}
	where the function \(F_{t_n(\t),\mu_n}\) denotes the inverse function of \(H_{t_n(\t),\mu_n}\) (see \eqref{def:Stieltjes} for its definition), mapping the real line to the graph of $y_{t_n(\t),\mu_n}$ (cf.~\cite{Biane}). 
		Indeed, any critical point $z$ of $\phi_{n,\tau}(z,v)$ satisfies
	\begin{align}
	z-x_n(\t)-\frac{v}{c_2 n^{2/3}}+t_n(\tau)G_{\mu_n}(z)=0,
	\end{align}
	or in other terms,
	\begin{align}
	H_{t_n(\t),\mu_n}(z)=x_n(\t)+\frac{v}{c_2 n^{2/3}},
	\end{align}
	so that inverting yields \eqref{def:znwn1}.
 For a more detailed discussion of the saddle points we refer to \cite[page 13]{ClaeysNeuschelVenker}, where  {convergence to} the sine kernel has been considered.\\
	Due to the rather complicated $n$-dependence of the saddle points, we will choose integration contours passing in a vicinity of the saddle points instead of going through these points exactly. This requires some finer knowledge on the location of the saddle points.
\begin{lemma}\label{lemma: saddle points}
	Suppose that Assumptions 1--3 hold, $G_2>0$ and $\kappa>2$. Then the following holds for $n\to\infty$.
	\begin{enumerate}
		\item \label{expansionH}We have uniformly for  $|z-x^*|= \mathcal{O}\left(n^{-\frac{1}{3}+\epsilon}\right)$ and $\t$ in compacts \[H_{t_n(\t),\mu_n}(z)=x_n(\t)+\frac{t_n(\t)G_2}{2} (z-x^*)^2 +\mathcal O\left(n^{-\frac{2}{3}+\epsilon}\right).\] 
		\item \label{expansionsaddles}We have	\begin{align*} 	z_n(\t_2,v)&=x^{*}+\mathcal{O}\left(n^{-\frac13+\frac\epsilon2}\right), \\
		w_n(\t_1,u)&= x^{*}+\mathcal{O}\left(n^{-\frac13+\frac\epsilon2}\right),
		\end{align*}
		where the  implied constants are uniform as $n\to\infty$ for \(\vert u\vert,\lv v\rv\leq K n^{\epsilon}\) for any constant \(K >0\) and $\t_1,\t_2$ in compacts.
	\end{enumerate}
\end{lemma}
\begin{proof} We first note that all $\O$-terms in the proof will be uniform for $\t$ in compacts as $n\to\infty$.

		We start with proving part (1) and write 
		\[H_{t_n(\t),\mu_n}(z)=H_{t_n(\t),\mu_n}(x^{*})+z-x^{*}+t_n(\t)\left(G_{\mu_n}(z)-G_{\mu_n}(x^{*})\right).\]
		By Proposition \ref{prop: Stieltjes comparison 2} (1), 
		for $|z-x^*|= \mathcal{O}\left(n^{-\frac{1}{3}+\epsilon}\right)$ {and by definition of $x_n(\t)=x_n^\Ai(\t)$ in \eqref{def:xnAiry}}, we get
		\[H_{t_n(\t),\mu_n}(z)
		=x_n(\t)+\left(1+t_n(\t) G_1\right)(z-x^{*})+\frac{t_n(\t)G_2}{2} (z-x^*)^2 + \O\left(n^{-\frac{2}{3}-2\e}\right).
		\]
		Now the result follows from
		\begin{align}
		t_n(\t) G_1=-\frac{t_n(\t)}{t_{\rm cr}}=-1+\O(n^{-1/3})\label{t_n_G_1}.
		\end{align}
		
		To prove part (2), we deal with the saddle points \(z_n\) and \(w_n\). Assume that $G_2>0$ and let \(x+iy\) be a point on the perimeter of the disk centered at \(x^{*}\) of radius \(n^{-1/3+ \epsilon}\) of the form 
		\[x+iy = x^{*}+n^{-1/3+\epsilon} e^{i\theta},\]
		where we choose an angle \(\delta < \theta < \frac{\pi}{2}-\delta\) for some small \(\delta >0\). Then we have \(y=\sin(\theta)n^{-1/3+\epsilon}\), and from Proposition \ref{prop: Stieltjes comparison 2} (1) and \eqref{t_n_G_1} we get
		\begin{align*}
		\Im G_{\mu_n}(x+iy) =& G_1 n^{-1/3+\epsilon} \sin(\theta) + \frac{G_2}{2} n^{-2/3+2\epsilon} \sin(2\theta)+ \O\left(n^{-2/3-2\e}\right)\\
		=& \left(-\frac{1}{t_n(\tau)}+\mathcal{O}(n^{-1/3})\right)y +G_2 y n^{-1/3+\epsilon} \cos(\theta)+ \O\left(n^{-2/3-2\e}\right),
		\end{align*}
		as \(n \to \infty\). Hence, we obtain for large \(n\)
		\[-\frac{\Im G_{\mu_n}(x+iy)}{y} = \frac{1}{t_n(\tau)}-G_2 \cos(\theta) n^{-1/3+\epsilon}\left(1+\mathcal{O}\left(n^{-\epsilon}\right)\right) < \frac{1}{t_n(\tau)}.\]
		We note that the latter inequality is valid uniformly in \(\delta < \theta < \frac{\pi}{2}-\delta\) for  any $0<\d<\pi/2$ and $n$ sufficiently large (which depends on $\d$). In a similar fashion, we obtain uniformly for angles \( \frac{\pi}{2}+\delta < \theta < \pi -\delta\) {(with $0<\d<\pi/2$, which we will assume from now on)} and large \(n\)
		\[-\frac{\Im G_{\mu_n}(x+iy)}{y} > \frac{1}{t_n(\tau)}.\]
		By definition of the function \(y_{t_n(\t),\mu_n}(x)\) {in \eqref{def:y} and
		\begin{align}
		-\frac{\Im G_{\mu_n}(x+iy)}{y}=\int\frac{d\mu_n(s)}{(x-s)^2+y^2},
		\end{align}} it follows that its graph
		lies for any $\t$ from a fixed compact set below the arc $\{x^*+n^{-1/3+\epsilon}e^{i\theta}:\delta<\theta<\pi/2-\delta\}$ for $x\in (x^*+n^{-1/3+\epsilon}\cos(\pi/2-\delta),x^*+n^{-1/3+\epsilon}\cos\delta)$   
		and above the arc $\{x^*+n^{-1/3+\epsilon}e^{i\theta}:\pi/2+\delta<\theta<\pi-\delta\}$ for $x\in (x^*-n^{-1/3+\epsilon}\cos\delta,x^*-n^{-1/3+\epsilon}\cos(\pi/2-\delta))$   .

		In other words, we have the following behavior for sufficiently large \(n\): if we consider {the closed disk $D$ centered at \(x^*\) with radius \(n^{-1/3+\epsilon}\), the graph of $y_{t_n(\t),\mu_n}$ enters the disk (not necessarily for the first time) at a point   $w_{\rm in}:=x^*+n^{-1/3+\epsilon}e^{i\theta_1}$ with $0\leq\theta_1<\delta$ and leaves the disk at a point
		$w_{\rm out}:=x^*+n^{-1/3+\epsilon}e^{i\theta_2}$ with $\pi/2-\delta<\theta_2<\pi/2+\d$. Between $w_{\rm in}$ and $w_{\rm out}$, the graph stays in the interior of the disk.  Setting
		\begin{align}\label{def:x_1,2}
		x_{\rm in,out}:= H_{t_n(\t),\mu_n}\left(w_{\rm in, out}\right) \in \mathbb{R}, 
		\end{align}
		part (\ref{expansionH}) shows
		\[x_{\rm in,out}= x_n(\t)+\frac{t_n(\t)}{2} G_2 n^{-2/3+2\epsilon} e^{2i \theta_{1,2} }+ \mathcal{O}\left(n^{-2/3 +\epsilon}\right),\quad n\to\infty.\]}
		From this we infer for sufficiently large \(n\) and positive constants \(k_0, k_1, k_2\)
		\[x_{\rm out} < x_n(\t) < x_{\rm in}\]
		with 
		\[x_n(\t)-x_{\rm out} \sim k_0 n^{-2/3 +2\epsilon}, \qquad x_{\rm in}-x_n(\t) \sim k_1 n^{-2/3 +2\epsilon},\]
		and 
		\begin{align*}
		x_{\rm in} -x_{\rm out} =\frac{t_n(\t)}{2}G_2 n^{-2/3+2\epsilon} \left(\cos(2\theta_1)-\cos(2\theta_2)\right)\left(1+\mathcal{O} \left(n^{-\epsilon}\right)\right) > k_2 n^{-2/3+2\epsilon}.
		\end{align*}
		The function \(F_{t_n(\t),\mu_n}\) maps the interval \([x_{\rm out}, x_{\rm in}]\) bijectively onto the part of the graph of \(y_{t_n(\t),\mu_n}\) that lies between $w_{\rm out}$ and $w_{\rm in}$.
		Moreover, using the expansion of part (\ref{expansionH}) for any point \(z\) on the graph of \(y_{t_n(\t),\mu_n}\) that lies between $w_{\rm in}$ and $w_{\rm out}$, we have with \(x= H_{t_n(\t),\mu_n}(z)\)
		\[[x_{\rm in}, x_{\rm out}] \ni x = H_{t_n(\t),\mu_n}(z)= x_n(\t)+ \frac{t_n(\t)}{2}G_2 \left(z-x^{*}\right)^2+\mathcal{O}\left(n^{-2/3+\epsilon}\right).\]
		If we now consider \(z\) as a function of \(x\) defined on \([x_{\rm out}, x_{\rm in}]\), which means \(z=F_{t_n(\t),\mu_n}(x)\), then we obtain
		\[z= x^{*}+\sqrt{\frac{2}{t_n(\t) G_2} (x-x_n(\t))+\mathcal{O}\left(n^{-2/3+\epsilon}\right)},\quad n\to\infty,\]
		where the implicit constant in the \(\mathcal{O}\)-term can be chosen to be independent of \(x\in[x_{\rm in}, x_{\rm out}]\) (and we have to choose the suitable branch of the square root).
		Hence, for \(x=x_n(\t)+\frac{v}{c_2n^{2/3}}\) with \(\vert v\vert \leq K n^{\epsilon}\), for some constant \(K>0\), we obtain for the saddle point \(z_n\)
		\[z_n = x^* +\sqrt{\frac{2}{t_n(\t_2) G_2} \frac{v}{c_2 n^{2/3}} + \mathcal{O}\left(n^{-2/3+\epsilon}\right)}=x^{*}+\mathcal{O}\left(n^{-1/3+\epsilon/2}\right),\quad n\to\infty.\]
		The assertion for the saddle point \(w_n\) follows by the same arguments.
	\end{proof}

\subsection{Choice of Contours}
We start by choosing the contour $\Gamma$ in \eqref{def:Kn}. As shown in \cite[page 14]{ClaeysNeuschelVenker}, the graph of $x\mapsto y_{t_n(\t_1),\mu_n}(x)$ is a descent path for $w\mapsto-\phi_{n,\t_1}(w,u)$ passing through the saddle point $w_n$. Moreover,  $x\mapsto y_{t_n(\t_1),\mu_n}(x)$ is of bounded support which contains all initial points $X_1(0),\dots,X_n(0)$. We consider the counterclockwise oriented graph of $y_{t_n(\t_1),\mu_n}$ restricted to the convex hull of its support, and we close this path by joining it with its complex conjugate. As Lemma \ref{lemma: saddle points} shows, the saddle points $z_n$ and $w_n$ both lie in an $\O(n^{-1/3+\e})$-neighborhood of $x^*$. To avoid complications of these merging points, we will make a local modification of the path defined above. To this end, we consider the closed disk $D$ of radius $n^{-1/3+\e}$ around $x^*$. As observed in the proof of Lemma \ref{lemma: saddle points}, there are $n$-dependent points $w_{\rm in}$ and $w_{\rm out}$ on the boundary of the disk in the upper half plane such that the graph of $y_{t_n(\t_1),\mu_n}$ enters the disk at $w_{\rm in}$, and leaves it at $w_{\rm out}$ while staying in the interior of the disk between those two points. Recall from that proof that for any small $\delta>0$, we have $\Re (w_{\rm in}-x^*)>0$ and $0\leq\arg(w_{\rm in}-x^*)<\d$ for $n$ sufficiently large, whereas $\pi/2-\d<\arg(w_{\rm out}-x^*)<\pi/2+\d$. We will assume {$0<\d<\pi/6$}.

We now set $$w_{1,n}:=w_{\rm in} \quad \text{ and }\quad w_{2,n}:=x^*+e^{i(\pi/2+\d)}n^{-1/3+\e}$$ with the same $\d>0$ as above. The point $w_{2,n}$ lies on the boundary of the disk and below the graph of $y_{t_n(\t_1),\mu_n}$ in the sense that with
\begin{align}
w_{3,n}:=\Re w_{2,n}+iy_{t_n,\mu_n}(\Re w_{2,n}),
\end{align}
we have $\Im w_{3,n}\geq\Im w_{2,n}$ for $n$ large enough. 
We now specify the $w$-integration contour $\Gamma$ by setting 
\begin{equation}\Gamma:=\Gamma_{\rm out}\cup \Gamma_{\rm in}\cup\overline{\Gamma_{\rm out}\cup \Gamma_{\rm in}},\end{equation} 
where $\Gamma_{\rm in}:= [w_{1,n},\Re w_{1,n}]\cup[\Re w_{1,n},x^*]\cup[x^*,w_{2,n}]$. The contour $\Gamma_{\rm out}$ is the union $\Gamma_{\rm out}:=\Gamma_{\rm out,1}\cup\Gamma_{\rm out,2}$, where $\Gamma_{\rm out, 1}$ is the part of the graph of $y_{t_n(\t_1),\mu_n}$ which lies  in the upper half plane and to the right of $w_{1,n}$, and  $\Gamma_{\rm out, 2}$ is the  union of the vertical line segment $[w_{2,n},w_{3,n}]$ with the part of the graph of $y_{t_n(\t_1),\mu_n}$ which lies  in the upper half plane and to the left of $w_{3,n}$. We choose the natural counter-clockwise orientation, see Figure \ref{fig:contourGamma} for a visualization of this contour.
\begin{figure}[h]
	\centering
	\footnotesize
	\def\svgwidth{0.7\columnwidth}
	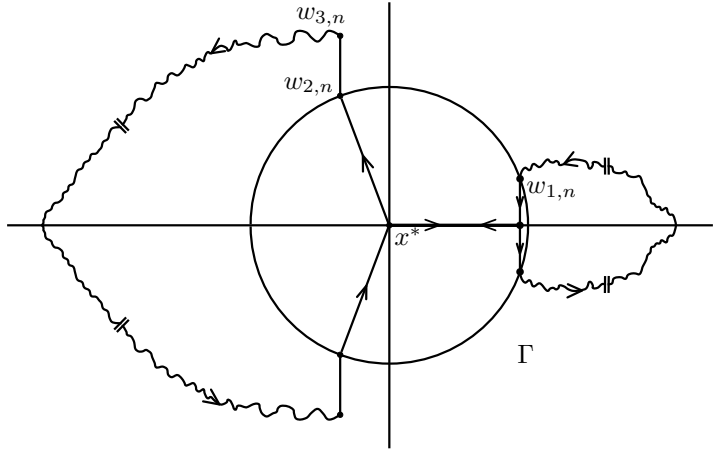
	\caption{Choice of contour $\Gamma$}\label{fig:contourGamma}
\end{figure}\\
For a modification of the $z$-integration contour $x_0+i\R$ in \eqref{def:Kn}, we define
\begin{align}\label{z_n_hat}
\widehat z_{n}:=x^*+e^{\frac{i\pi}{3}}n^{-1/3+\e},
\end{align}
which lies on the boundary of the disk \(D\). With this definition we modify $x_0+i\R$ to the contour $\Sigma$ given as (see Figure \ref{fig:contourSigma})
\begin{equation}
\Sigma:=\Sigma_{\rm in}\cup\Sigma_{\rm out} \cup \overline{\Sigma_{\rm in}\cup\Sigma_{\rm out}},\qquad 
\Sigma_{\rm out}:=[\widehat z_n, \widehat z_n+i\infty],\qquad
\Sigma_{\rm in}:=[x^*,\widehat z_n].
\end{equation}
\begin{figure}[h]
	\centering
	\def\svgwidth{0.5\columnwidth}
	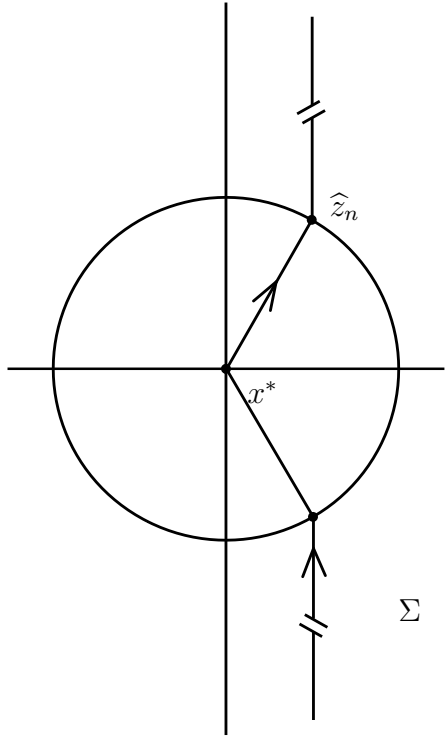
	\caption{Choice of contour $\Sigma$}\label{fig:contourSigma}
\end{figure}
{For simplicity of notation, we restrict our analysis of the correlation kernel to the case $\t_2\geq\t_1$, in which the heat kernel part, i.e.\ the second term in \eqref{def:Kn}, is not present.  It is straightforward to check that, for $\t_2<\t_1$, plugging the rescaled variables into the heat kernel in \eqref{def:Kn} and adding the gauge factors given  {in \eqref{gauge_factor}} yields the heat kernel part of the extended Airy kernel.

From \eqref{def:Kn} and \eqref{conjugation} we obtain 
\begin{multline}
\frac{1}{c_2 n^{2/3}} \tilde{K}_{n,t_n(\t_1),t_n(\t_2)}\left(x_n(\t_1)+\frac{u}{c_2 n^{2/3}}, x_n(\t_2)+\frac{v}{c_2 n^{2/3}}\right)\\\label{rescaledkernel}
=  \frac{n^{1/3}e^{f_n(\t_2,v)-f_n(\t_1,u)}}{c_2 (2\pi i)^2 \sqrt{t_n(\t_1)t_n(\t_2)}} \int_\Sigma dz \int_{\Gamma} dw ~ \frac{e^{\phi_{n,\t_2} (z,v) - \phi_{n,\t_1} (w,u)}}{z-w}, 
\end{multline}
where $\phi_{n,\t}$ has been defined in \eqref{def:phi} and we write in a slight abuse of notation $f_n(\t,u)$ instead of
\[f_n\lb t_n(\t),x_n(\t)+\frac{u}{c_2n^{2/3}}\rb
\]
with $f_n$ from \eqref{gauge_factor}. 
Note that the integration contours for $z$ and $w$ in \eqref{rescaledkernel} have an intersection point at $z=w=x^*$ and hence do not satisfy the conditions imposed in \eqref{def:Kn}. However, observing that the singularity is integrable, we easily derive \eqref{rescaledkernel} from \eqref{def:Kn} by taking a limit.

We split the  double integral \eqref{rescaledkernel} into the three parts
\begin{align}
&\label{def:Kn1}{K}_{n,\t_1,\t_2}^{(1)}(u,v):=\frac{n^{1/3}e^{f_n(\t_2,v)-f_n(\t_1,u)}}{c_2 (2\pi i)^2 \sqrt{t_n(\t_1)t_n(\t_2)}} \int_{\Sigma_{\rm in}\cup \overline{\Sigma_{\rm in}}} dz \int_{\Gamma_{\rm in}\cup\overline{\Gamma_{\rm in}}} dw ~ \frac{e^{\phi_{n,\t_2} (z,v) - \phi_{n,\t_1} (w,u)}}{z-w},\\
&\label{def:Kn2}{K}_{n,\t_1,\t_2}^{(2)}(u,v):=\frac{n^{1/3}e^{f_n(\t_2,v)-f_n(\t_1,u)}}{c_2 (2\pi i)^2 \sqrt{t_n(\t_1)t_n(\t_2)}} \int_{\Sigma} dz \int_{\Gamma_{\rm out}\cup\overline{\Gamma_{\rm out}}} dw ~ \frac{e^{\phi_{n,\t_2} (z,v) - \phi_{n,\t_1} (w,u)}}{z-w},\\
&\label{def:Kn3}{K}_{n,\t_1,\t_2}^{(3)}(u,v):=\frac{n^{1/3}e^{f_n(\t_2,v)-f_n(\t_1,u)}}{c_2 (2\pi i)^2 \sqrt{t_n(\t_1)t_n(\t_2)}} \int_{\Sigma_{\rm out}\cup\overline{\Sigma_{\rm out}}} dz \int_{\Gamma_{\rm in}\cup\overline{\Gamma_{\rm in}}} dw ~ \frac{e^{\phi_{n,\t_2} (z,v) - \phi_{n,\t_1} (w,u)}}{z-w}.\quad
\end{align}

\subsection{Analysis of the main term}
The main term giving the extended Airy kernel in the limit is \({K}_{n,\t_1,\t_2}^{(1)}\) as the following proposition shows.
\begin{prop}\label{prop: int1} For any constant $\s>0$ we have
	\begin{equation*} {K}_{n,\t_1,\t_2}^{(1)}(u,v)=\mathbb K^{\rm Ai}_{\t_1,\t_2}\left(u, v\right)+\O\lb \frac{e^{-\s(u+v)}}{n^{\e}}\rb,
	\end{equation*}
	as \(n\to\infty\), where the $\O$ term is uniform in $u,v\in [-M, Mn^{\epsilon}]$ for any fixed $M>0$.
\end{prop}

\begin{proof}All $\O$-terms will be uniform for $\t_1,\t_2$ in compacts as $n\to\infty$.
 Note that the $w$-integral can be restricted to $[w_{1,n},\Re w_{1,n}]\cup[x^*,w_{2,n}]$ and its complex conjugate, as the two integrals over the interval $[x^*,\Re w_{1,n}]$ cancel each other.

 We will reparametrize $z$ and $w$ by 
 \begin{align}
 z=:x^{*}+\frac{c_2t_{\rm cr}\zeta}{n^{1/3}}\quad \text{ and } \quad w=:x^{*}+\frac{c_2t_{\rm cr}\w}{n^{1/3}}.\label{change_of_variables}
 \end{align} With this parametrization and Proposition \ref{prop: Stieltjes comparison 2}, we have uniformly for $z$ in the disk \(D\) (which is equivalent to \(\zeta = \mathcal{O}\left(n^{\epsilon }\right)\))
\begin{align}
&ng_{\mu_n}(z)=ng_{\mu_n}(x^*)+n\int_{x^*}^zd\xi~ G_{\mu_n}(\xi)\\
&=ng_{\mu_n}(x^*)+n^{2/3}G_0c_2t_{\rm cr}\zeta+\frac{n^{1/3}G_1c_2^2t_{\rm cr}^2\z^2}2+\frac{G_2c_2^3t_{\rm cr}^3\z^3}6+{\O\lb\frac{\lv\zeta\rv}{n^{2\e}}\rb}.\label{calculation:g}
\end{align} 
Using 
\begin{align}
\frac{t_{\rm cr}}{t_n(\t_2)}=1+\O\lb\frac{\t_2}{n^{1/3}}\rb
\end{align}
and
\begin{align}
 &\frac{n^{1/3}c_2^2t_{\rm cr}^2\z^2}2 \lb G_1+\frac1{t_n(\t_2)}\rb= \frac{n^{1/3}c_2^2t_{\rm cr}^2\z^2}2\lb-\frac{1}{t_{\rm cr}}+\frac1{t_n(\t_2)}\rb\\
 &=-\z^2\t_2 \frac{t_{\rm cr}}{t_n(\t_2)}=-\z^2\t_2+\O\lb \frac{1+\lv \t_2\rv+\lv\z\rv^2}{n^{1/3}}\rb,
\end{align}
 we obtain after an elementary computation from \eqref{def:phi} and \eqref{calculation:g} 
\begin{align}
\phi_{n,\t_2}(z,v)=\frac{\z^3}3-\t_2\z^2-v\z +\frac{nt_n(\t_2)G_0^2}2+\frac{n^{1/3}G_0v}{c_2}+ng_{\mu_n}(x^*)+{\O\lb n^{-\e}\rb} \label{calculation:phi}
\end{align}
{with the $\O$-term being uniform for $\lv\zeta\rv,\lv v\rv\leq Kn^{\e}$}, \(K>0\) being an absolute constant.
A similar computation with $\phi_{n,\t_1}(w,u)$ yields
\begin{align}
&\phi_{n,\t_2}(z,v)-\phi_{n,\t_1}(w,u)+f_n(\t_2,v)-f_n(\t_1,u)\\
&=\frac{\z^3}3-\t_2\z^2- v\z-\frac{\w^3}3+\t_1\w^2+ u\w+{\O\lb n^{-\e}\rb},\label{phase_function_local}
\end{align}
{with the $\O$-term now being uniform for $\lv\zeta\rv,\lv \w\rv,\lv v\rv,\lv u\rv\leq Kn^{\e}$.}
With the changes of variables $z\mapsto \zeta$ and $w\mapsto \omega$, the remaining constant in front of the double integral reduces to
\begin{align}
\frac{t_{\rm cr}}{(2\pi i)^2\sqrt{t_n(\t_1)t_n(\t_2)}}=\frac{1}{(2\pi i)^2}+\O\lb \frac{\lv \t_1\rv+\lv \t_2\rv}{n^{1/3}}\rb.
\end{align}
{Thus there is a function $h_n(\t_1,\t_2,\z,\w,u,v)={\O\lb n^{-\e}\rb}$ (uniformly in all relevant parameters) such that we can write}
\begin{align}
&{K}_{n,\t_1,\t_2}^{(1)}(u,v)\\
&=\lb\frac{1}{(2\pi i)^2}+\O\lb \frac{\lv \t_1\rv+\lv \t_2\rv}{n^{1/3}}\rb\rb\int_{\widehat\Sigma} d\z \int_{\widehat\Gamma} d\w ~ \frac{e^{\frac{\z^3}3-\t_2\z^2- v\z-\frac{\w^3}3+\t_1\w^2+u\w +h_n(\t_1,\t_2,\z,\w,u,v)}}{\z-\w},\quad\quad\label{K_n_1_beginning}
\end{align}
where $\widehat\Sigma$ and $\widehat\Gamma$ are the contours $\Sigma_{\rm in}$ and $\Gamma_{\rm in}$ under the change of variables \eqref{change_of_variables}, i.e.~
\[\widehat\Sigma:=\left[\frac{e^{-i\frac{\pi}{3}}n^\epsilon}{c_2 t_{\rm cr}},0\right]\cup\left[0,\frac{e^{i\frac{\pi}{3}}n^\epsilon}{c_2 t_{\rm cr}}\right],\] and $\widehat\Gamma:=\widehat\Gamma_{\rm left}\cup\widehat\Gamma_{\rm right}$ with 
\[
\widehat\Gamma_{\rm left}:=\left[-\frac{ie^{-i\delta}n^{\epsilon}}{c_2 t_{\rm cr}},0\right]\cup\left[0,\frac{ie^{i\delta}n^{\epsilon}}{c_2 t_{\rm cr}}\right],\quad \widehat\Gamma_{\rm right}:=\left[\frac{n^{1/3}}{c_2 t_{\rm cr}}(w_{1,n}-x^*), \overline{\frac{n^{1/3}}{c_2 t_{\rm cr}}(w_{1,n}-x^*)}\right].\]

{We will show first that 
\begin{equation}\label{attemptLemma4.1_2}\int_{\widehat\Sigma} d\zeta \int_{\widehat\Gamma_{\text{right}}} d\omega \frac{e^{\frac{\z^3}3-\t_2\z^2-v\z -\frac{\w^3}3+\t_1\w^2+u\w +h_n(\t_1,\t_2,\z,\w,u,v)}}{\zeta -\omega} =\mathcal{O}\left(e^{-dn^{3\e}}\right),
\end{equation}
for some $d>0$, uniformly in $u,v\in[-M,Mn^\e]$ and $\t_1,\t_2$ from compacts. To see this, we take absolute values and use the bounds
\begin{align}\label{decoupling}
&\frac{1}{\vert \zeta - \omega\vert} \leq \frac{K}{n^{\epsilon}},\quad\lv\exp(h_n)\rv\leq C,
\end{align}
valid for \(\zeta\) on \(\widehat\Sigma\), \(\omega\) on \(\widehat\Gamma_{\text{right}}\) {and} $C,K>0$ are constants. This gives as an upper bound for the left-hand side of \eqref{attemptLemma4.1_2}
\[CKn^{-\epsilon} \int_{\widehat\Sigma} \vert d\zeta\vert \left\vert e^{\frac{\z^3}3-\t_2\z^2-v\z }\right\vert
\int_{\widehat\Gamma_{\text{right}}} \vert d\omega\vert 
\left\vert e^{-\frac{\w^3}3+\t_1\w^2+u\w }\right\vert. \]
Now, since $0\leq\arg(w_{1,n}-x^*)<\d<\frac{\pi}{6}$, we have for $\w\in \widehat{\Gamma}_{\text{right}}$ and $u\leq M n^\e$ for any $M>0$ that $\Re (-\w^3/3+\t_1\w^2+u\w )<-dn^{3\e}$ for some $d>0$. As the length of the contour $\widehat \Gamma_{\text{right}}$  {is $\O(n^{\e})$}, this proves \eqref{attemptLemma4.1_2} upon noting that the $\z$-integral is uniformly bounded in $n$ for $v>-M$ and any $M>0$.
}

{In the next step, we will get rid of the remainder $h_n$ in the exponent of the right-hand side of \eqref{K_n_1_beginning}, where we are now allowed to replace $\widehat{\Gamma}$ by $\widehat{\Gamma}_{\rm left}$. With the inequality  $\lv 1-\exp(h_n)\rv\leq \lv h_n\rv\exp\lv h_n\rv$, we have for any constant $\s>0$ 
	\begin{align}\label{K_n^1_second_last}
	&e^{ \s(u+v)}\Bigg\vert \int_{\widehat\Sigma} d\z \int_{\widehat\Gamma_{\rm left}} d\w ~ \frac{e^{\frac{\z^3}3-\t_2\z^2-\z v-\frac{\w^3}3+\t_1\w^2+u\w }}{\z-\w}\\
	&\qquad-\int_{\widehat\Sigma} d\z \int_{\widehat\Gamma_{\rm left}} d\w ~ \frac{e^{\frac{\z^3}3-\t_2\z^2-\z v-\frac{\w^3}3+\t_1\w^2+u\w +h_n(\t_1,\t_2,\z,\w,u,v)}}{\z-\w}\Bigg\vert\\
	&\leq \int_{\widehat\Sigma} d\lv\z\rv \int_{\widehat\Gamma_{\rm left}} d\lv\w\rv ~ \frac{\left\lv h_n(\t_1,\t_2,\z,\w,u,v)\right\rv\left\lv e^{\frac{\z^3}3-\t_2\z^2-v(\z- \s) -\frac{\w^3}3+\t_1\w^2+u(\w+ \s) +\lv h_n(\t_1,\t_2,\z,\w,u,v)\rv}\right\rv}{\lv\z-\w\rv}.\label{K_n_1_estimate}
	\end{align}
As the integral
\begin{align}
\int_{\widehat\Sigma} d\lv\z\rv \int_{\widehat\Gamma_{\rm left}} d\lv\w\rv ~ \frac{\lv e^{\frac{\z^3}3-\t_2\z^2-v(\z-\s) -\frac{\w^3}3+\t_1\w^2+u(\w+ \s) }\rv}{\lv\z-\w\rv}
\end{align}
 is uniformly bounded in $n$ and $\lv h_n(\t_1,\t_2,\z,\w,u,v)\rv={\O\lb n^{-\e}\rb}$, the r.h.s.~of \eqref{K_n_1_estimate} is $\O(n^{-\e})$.
}

A standard argument now shows that extending the contours $\widehat \Sigma$ and 
\begin{align}
\left[-\frac{ie^{-i\delta}n^{\epsilon}}{c_2 t_{\rm cr}},0\right]\cup\left[0,\frac{ie^{i\delta}n^{\epsilon}}{c_2 t_{\rm cr}}\right]
\end{align}
to $\Sigma^\Ai$ and $ie^{i\delta}\mathbb R^+ \cup -ie^{-i\delta}\mathbb R^+$ {yields an exponentially small error}. Finally, by analyticity, the contour $ie^{i\delta}\mathbb R^+ \cup -ie^{-i\delta}\mathbb R^+$ can be deformed to $\Gamma^\Ai$. Combining the above estimates, we obtain the result.
\end{proof}

\subsection{Analysis of the remaining terms}
In this subsection, we discuss the necessary asymptotic estimates of the double integrals $K_{n,\t_1,\t_2}^{(i)}$ for $i=2,3$. We will employ the descent properties  of {$\phi_{n,\t}$} on the chosen contours, which are the content of the following lemma. For its statement, recall the definition of the saddle points $z_n$ and $w_n$ in \eqref{def:znwn}.
\begin{lemma}\label{lemma:realpartphi}\leavevmode
	\begin{enumerate}\item For any $u,\t\in\mathbb R$ and \(n \in \mathbb N\), the function \[\mathbb R\ni\alpha\mapsto \Re\phi_{n,\t}\left(\alpha+iy_{t_n(\t),\mu_n}(\alpha),u\right)\] achieves its unique local minimum at $\alpha=\Re w_n$, is strictly decreasing for $\alpha<\Re w_n$ and strictly increasing for $\alpha>\Re w_n$.
		\item For any $\alpha, v,\t\in\mathbb R$ and \(n\in \mathbb N\), the function 
		\[[0,\infty)\ni\beta\mapsto \Re\phi_{n,\t}\left(\alpha+i\beta,v\right)\]
	 achieves its unique local maximum at $\beta=y_{t_n(\t),\mu_n}(\alpha)$, is strictly increasing for $0<\beta<y_{t_n(\t),\mu_n}(\alpha)$ and strictly decreasing for $\beta>y_{t_n(\t),\mu_n}(\alpha)$.
	\end{enumerate}
\end{lemma}
\begin{proof}
	In order to see the first part of the lemma, we compute
	\[\frac{\partial}{\partial \alpha} \Re\phi_{n,\t}\left(\alpha+iy_{t_n(\t),\mu_n}(\alpha),u\right) =\frac{n}{t_n(\t)}\left(H_{t_n(\t),\mu_n}\left(\alpha+ i y_{t_n(\t),\mu_n}(\alpha)\right)-x_n(\t)-\frac{u}{c_2 n^{2/3}} \right)\]
	and use that
	\[\alpha \mapsto H_{t_n(\t),\mu_n}\left(\alpha+ i y_{t_n(\t),\mu_n}(\alpha)\right)\]
	constitutes a bijective mapping from \(\mathbb{R}\) to \(\mathbb{R}\). 
	
	For the second part we compute
	\[\frac{\partial}{\partial \beta} \Re\phi_n\left(\alpha+i\beta,v\right)=\frac{n\beta}{t_n(\t)}\left[t_n(\t) \int \frac{d\mu_n(s)}{(\alpha-s)^2+\beta^2} -1\right],\]
	which is positive for \(0<\beta < y_{t_n(\t),\mu_n}(\alpha)\), zero at \(\beta = y_{t_n(\t),\mu_n}(\alpha)\) and negative for \(\beta>y_{t_n(\t),\mu_n}(\alpha)\).
\end{proof}
With the preceding lemma at hand, we can now quantify growth and decay of the phase function when leaving the interior contours $\Gamma_{\rm in}$ and $\Sigma_{\rm in}$, respectively.
\begin{lemma}\label{lemma:auxiliaryremaining}
There is a constant $d>0$ such that for $n$ large enough the following holds uniformly for $\t_1,\t_2$ in compacts and $u,v\in[-M,Mn^\e]$:
\begin{enumerate}
\item\label{remaining2} For $z\in \Sigma,\ w\in\Gamma_{\rm out}\cup\overline{\Gamma_{\rm out}}$  we have
\begin{align}\label{remaining1}
\Re\left(\phi_{n,\t_2}(z,v)-\phi_{n,\t_1}(w,u)+f_n(\t_2,v)-f_n(\t_1,u)\right)\leq -dn^{3\e}.
\end{align}
\item  For $z\in\Sigma_{\rm out}\cup\overline{\Sigma_{\rm out}},\ w\in\Gamma_{\rm in}\cup\overline{\Gamma_{\rm in}}$ we have
\begin{align}\label{eq1_lemma_decay}
\Re\left(\phi_{n,\t_2}(z,v)-\phi_{n,\t_1}(w,u)+f_n(\t_2,v)-f_n(\t_1,u)\right)\leq -dn^{3\e}.
\end{align}
\end{enumerate}
\end{lemma}
\begin{proof}
Without loss of generality, we restrict our attention to the contours in the upper half plane. To see part (1) of the Lemma, we argue that by \eqref{phase_function_local}, equation \eqref{remaining1} holds for $z\in\Sigma_{\rm in}$, $w\in\{w_{1,n},w_{2,n}\}$ and $n$ large enough. Indeed, $\Sigma_{\rm in}$ and $w\in \{w_{1,n},w_{2,n}\}$ lie in the sectors where $\Re \z^3<0$ and $\Re \w^3>0$, respectively. The remaining terms in \eqref{phase_function_local} are $\O(n^{2\e})$. 
To extend this to $z\in\Sigma$ and $w\in\Gamma_{\rm out}$, we observe that by Lemma \ref{lemma:realpartphi} (2), $\Re\phi_{n,\t_2}(z,v)$ decreases along $\Sigma_{\rm out}$ as we move from $\Sigma_{\rm in}$ to $\Sigma_{\rm out}$. This follows from the graph of $y_{t_n(\t_2),\mu_n}$ lying below the boundary point $\widehat{z}_n$ (see \eqref{z_n_hat} for the definition of $\widehat z_n$). On the other hand we have 
\begin{align}
\Re(-\phi_{n,\t_1}(w,u))\leq \max\lb\Re(-\phi_{n,\t_1}(w_{1,n},u)),\Re(-\phi_{n,\t_1}(w_{2,n},u))\rb,
\end{align}
since by Lemma \ref{lemma:realpartphi} (2), $\Re(\phi_{n,\t_1}(w,u))$ increases from $w_{2,n}$ to $w_{3,n}$ along $\Gamma_{\rm out}$ (as the graph of $y_{t_n(\t_1),\mu_n}$ lies above $w_{2,n}$) and further increases when following $y_{t_n(\t_1),\mu_n}$ to the left according to Lemma \ref{lemma:realpartphi} (1). By the same argument, $\Re(\phi_{n,\t_1}(w,u))$ increases when following $y_{t_n(\t_1),\mu_n}$ from $w_{1,n}$ to the right. In this reasoning, we used that the saddle point $w_n$ lies inside the disk $D$ around $x^*$ by Lemma \ref{lemma: saddle points} (2). From this we conclude part (1) of Lemma \ref{lemma:auxiliaryremaining}.
The second part (2) follows by analogous arguments. 
\end{proof}

We are now prepared to estimate the remaining terms ${K}_{n,\t_1,\t_2}^{(2)}$ and ${K}_{n,\t_1,\t_2}^{(3)}$ in \eqref{def:Kn2}-\eqref{def:Kn3}.
\begin{lemma}\label{lemma: int2}
As $n\to\infty$, we have for $j=2,3$ and some constant $D>0$ and any $\s>0$
\begin{equation}{K}_{n,\t_1,\t_2}^{(j)}(u,v)=\mathcal O(e^{-n^{D}-\s(u+v)})\label{remainder_exponential}
\end{equation}
uniformly for $\tau_1,\t_2$ in compact subsets of $\mathbb R$, and for $u,v\in [-M, Mn^{\epsilon}]$ for any $M>0$.
\end{lemma}
\begin{proof}
To bound the double contour integral in the definition of ${K}_{n,\t_1,\t_2}^{(j)}(u,v)$, $j=2,3$, we first take absolute values and use $\lv z-w\rv\geq C n^{-1/3+\e}$ on the respective contours for some $C>0$ sufficiently small. The remaining integrals can then  be estimated using Lemma \ref{lemma:auxiliaryremaining}. For $j=2$, we split the integral over $\Sigma$ into two parts, corresponding to $\Sigma\cap\{\lv z\rv\leq n\}$ and $\Sigma\cap\{\lv z\rv> n\}$, respectively. The integral over $\Sigma\cap\{\lv z\rv\leq n\}\times \Gamma_{\rm out}$ gives by Lemma \ref{lemma:auxiliaryremaining} (1) a contribution of at most $\O(e^{-n^{3\e-\delta'}})$ for any $0<\d'<\epsilon$. Here we used that by \cite[Lemma 2.1]{ClaeysNeuschelVenker}, the length of $\Gamma_{\rm out}$ 
is $\O(n^3)$ uniformly for $\t_1$ in compacts. 

For the integral over $\Sigma\cap\{\lv z\rv> n\}\times \Gamma_{\rm out}$ we use that for $n$ large enough by Lemma \ref{lemma:auxiliaryremaining} part (1) we know 
\begin{align}\label{phase_function_out}
\Re\left(\phi_{n,\t_2}(z,v)-\phi_{n,\t_1}(w,u)+f_n(\t_2,v)-f_n(\t_1,u)\right)\leq 0
\end{align}
for $(z,w)\in\Sigma\times\Gamma_{\rm out}$. On the other hand, for sufficiently large $n$ we have

\begin{align*}
-\frac n{t_{\rm cr}}(\Im z)^2\leq\Re\phi_{n,\t_2}(z,v)=\Re\lb\frac{n}{2t_n(\t_2)} \left(z-x_n(\t_2)-\frac{v}{c_2 n^{2/3}}\right)^2 + ng_{\mu_n}(z)\rb\leq -\frac n{4t_{\rm cr}}(\Im z)^2,
\end{align*}
$z\in \Sigma\cap\{\lv z\rv>n/3\}$, $v\in[-M,Mn^\e]$, $\t_2$ in a compact set. This follows from the logarithmic growth of $g_n(z)$ as $z\to\infty$.
We conclude that for $n$ large enough we have
\begin{align}
&\Re\left(\phi_{n,\t_2}(z,v)-\phi_{n,\t_1}(w,u)+f_n(\t_2,v)-f_n(\t_1,u)\right)\\
&=\Re\left(\phi_{n,\t_2}(z/3,v)-\phi_{n,\t_1}(w,u)+f_n(\t_2,v)-f_n(\t_1,u)\right)+\Re\lb\phi_{n,\t_2}(z,v)\rb-\Re\left(\phi_{n,\t_2}(z/3,v)\right)\\
&\leq 0-\lb\frac{1}{4}-\frac19\rb\frac n{t_{\rm cr}}(\Im z)^2=-\frac {5n}{36t_{\rm cr}}(\Im z)^2.
\end{align}
This shows that the integral over $\Sigma\cap\{\lv z\rv> n\}\times \Gamma_{\rm out}$ is of order $\O(\exp({-\tilde d n^3}))$ for some $\tilde d>0$, which finishes the proof for $j=2$.

For $j=3$, assertion \eqref{remainder_exponential} follows from analogous arguments using Lemma \ref{lemma:auxiliaryremaining} (2).
\end{proof}

\subsection{Proofs of Theorems \ref{thm:main} and \ref{corollary_TW}}

\begin{proof}[Proof of Theorem \ref{thm:main}]
Theorem \ref{thm:main} follows by combining Proposition \ref{prop: int1} and Lemma \ref{lemma: int2}.
\end{proof}

\begin{proof}[Proof of Theorem \ref{corollary_TW}]
	{We start with proving assertion (2).}
Let $\t_1<\cdots<\t_m$ and $a_1,\dots,a_m$ be given. The distribution function of $\xi(\tau_1), \ldots, \xi(\tau_m)$ is given as a gap probability,
\begin{align}
&\P\left( c_2n^{2/3}(\xi(\t_1)-x_n(\t_1))\leq a_1,\dots,c_2n^{2/3}(\xi(\t_m)-x_n(\t_m))\leq a_m\right)\\
&=\P\lb \bigcap_{j=1}^m \text{no eigenvalue at time $\t_j$ in }\Big(x_n(\t_j)+\frac{a_j}{c_2n^{2/3}},x_n(\t_j)+\frac{n^\e}{n^{2/3}}\Big]\rb,\label{Fredholm_representation}
\end{align}
which in turn can be expressed, as is well-known (cf.~\cite{Johansson03}), by inclusion-exclusion as a Fredholm determinant
\begin{align}
&\det(I-\hat{\mathcal{K}}_{n,a_1,\dots,a_m})_{L^2(\{\t_1,\dots,\t_m\}\times\R,\#\otimes \lambda)}\\
&:=\sum_{k=0}^\infty \frac{(-1)^k}{k!}\sum_{i_1,\dots,i_k=1}^m\int_{\R^k}\det\left[1_{[a_{i_j},c_2n^\e]}(u_{i_{j}})\hat K_{n,\t_{i_j},\t_{i_{j'}}}(u_{i_j},u_{i_{j'}})1_{[a_{i_{j'}},c_2n^\e]}(u_{i_{j'}})\right]_{1\leq j,j'\leq k}du_{i_1}\dots du_{i_k}.\label{Fredholm_explicit}
\end{align}
Here we recall that $\#$ and $\lambda$ are counting and Lebesgue measure, respectively, and
\begin{align*}
\hat K_{n,\t,\t'}(u,u'):=\frac{1}{c_2n^{2/3}}\tilde K_{n,t_{n}^{\rm Ai}(\t),t_{n}^{\rm Ai}(\t')}\left(x^{\Ai}_n(\t)+\frac{u}{c_2n^{2/3}}, x^{\Ai}_n(\t')+\frac{u'}{c_2n^{2/3}}\right)
\end{align*}
is the rescaled kernel of Theorem \ref{thm:main}. Finally, $\hat{\mathcal{K}}_{n,a_1,\dots,a_m}$ is the integral operator on $L^2(\{\t_1,\dots,\t_m\}\times\R,\#\otimes \lambda)$ with kernel
\begin{align}
r_n(\t,u)\hat K_{n,\t,\t'}(u,u')r_n(\t',u'),
\end{align}
where $r_n(\t_j,u):=1_{[a_j,c_2n^\e]}(u)$, $j=1,\dots,m$.

In view of \eqref{Airy_2_def}, to prove the theorem, we need to show convergence of a sequence of Fredholm determinants to another Fredholm determinant. To this end, we use a convenient comparison result from \cite{AGZ} which we first state in (almost) full generality:
let $X$ be a locally compact Polish space and $\nu$ a finite measure with total mass $\|\nu\|_1$ on $(X,\mathcal B(X))$, where $\mathcal B(X)$ is the Borel $\sigma$-algebra of $X$. Let $K_1,K_2:X\times X\to\R$ be uniformly bounded. Then the Fredholm determinants $\Delta(\mathcal K_i):=\det(I-\mathcal K_i)_{L^2(X,\nu)},\ i=1,2$ exist, where $\mathcal K_i$ denotes the integral operator with kernel $K_i$, $i=1,2$, and we have the inequality \cite[Lemma 3.4.5]{AGZ}
\begin{align}\label{Fredholm_comparison}
\lv\Delta(\mathcal K_1)-\Delta(\mathcal K_2)\rv\leq \lb\sum_{k=1}^\infty\frac{k^{1+k/2}\|\nu\|_1^k\max(\|K_1\|_\infty,\|K_2\|_\infty)^{k-1}}{k!}\rb\|K_1-K_2\|_\infty, \qquad
\end{align}
where $\|\cdot\|_\infty$ denotes the sup norm on $X\times X$.

In our application of \eqref{Fredholm_comparison}, we choose $X:=\{\t_1,\dots,\t_m\}\times[\min_{j=1}^m a_j,\infty)$. On $\{\t_1,\dots,\t_m\}$ we choose the counting measure $\#$ as before but to overcome the non-finiteness of the Lebesgue measure $\lambda$, we use the fast decay of the kernels $\hat{K}_{n,\t,\t'}$ and its limit. Define $\hat \lambda(du):=e^{-2u}\lambda(du)$, $\nu:=\#\otimes\hat\lambda$ and
\begin{align}
&K_1((\t,u),(\t',u')):=K_{1,n}((\t,u),(\t',u')):=r_n(\t,u)\hat K_{n,\t,\t'}(u,u')e^{u+u'}r_n(\t',u'),\quad\label{def_K_1}\\
&K_2((\t,u),(\t',u')):=r(\t,u)\mathbb K^{\rm Ai}_{\t,\t'}(u,u')e^{u+u'}r(\t',u'),\label{def_K_2}
\end{align}
where $r(\t,u) =$ has been defined following \eqref{def_Fredholm_limit}. It follows from the explicit representation of Fredholm determinants (see \eqref{Fredholm_explicit}) that 
\begin{align*}
&\det(I-\hat{\mathcal{K}}_{n,a_1,\dots,a_m})_{L^2(\{\t_1,\dots,\t_m\}\times\R,\#\otimes \lambda)}=\Delta(\mathcal K_1),\\
&\det(I-\mathcal K^{\rm Ai}_{a_1,\dots,a_m} )_{L^2(\{\t_1,\dots,\t_m\}\times\R,\#\otimes \lambda)}=\Delta(\mathcal K_2).
\end{align*} 
Now, it follows from the arguments at the end of the proof of Proposition \ref{prop: int1} (beginning with \eqref{K_n^1_second_last}) that $K_1$ and $K_2$ are uniformly bounded  in $n$ on $X\times X$.
	By Theorem \ref{thm:main}, we have
	\begin{align*}
	\lv K_1((\t,u),(\t',u'))-K_2((\t,u),(\t',u'))\rv={\O(n^{-\e})},\quad n\to\infty,
	\end{align*} 
	for $\t,\t'\in\{\t_1,\dots,\t_m\}$ and $u,u'\in[\min_{j=1}^ma_j,c_2n^{\e}]$. If at least one of $u$ or $u'$ is larger than $c_2n^\e$, then $K_1((\t,u),(\t',u'))=0$ and the exponential decay of $K_{\t,\t'}^\Ai(u,u')$ {(which is implied by the boundedness of $K_2$)}
	 shows
	\begin{align}
	\lv K_1((\t,u),(\t',u'))-K_2((\t,u),(\t',u'))\rv=\O\lb\frac{e^{-d(u+u')}}{n^\e}\rb\label{comparison_kernels}
	\end{align}
	for some $d>0$. Summarizing, we get $\|K_1-K_2\|_\infty={\O(n^{-\e})}$, which together with \eqref{Fredholm_comparison} finishes the proof of part (2).
	
	For part (1), we will show that for $n$ large enough
	\begin{align}
	&\left\lv\P\lb \bigcap_{j=1}^m \text{no eigenvalue at time $\t_j$ in }\left[ x_n^\Ai(\t_j)+n^{\e'-\frac23},x_n^\Ai(\t_j)+n^{\e-\frac23}\right]\rb-1\right\rv\\
	&\leq\lv\Delta(\mathcal K_1)-\Delta(\mathcal K_2)\rv+\lv\Delta(\mathcal K_2)-\Delta(\mathcal K_3)\rv\leq e^{-n^{\d}},\label{Fredholm_inequalities}
	\end{align}
with $\e, \e'$ as in the statement of the theorem, where $\mathcal K_1,\mathcal K_2$ are the integral operators on $L^2(X,\#\otimes\hat\lambda)$ with  kernels \eqref{def_K_1} and \eqref{def_K_2} on $L^2(X,\#\otimes\hat\lambda)$ as above, but now with $r_n(\t,u):=r(\t,u):=1_{[c_2n^{\e'},c_2n^\e]}(u)$ on $X:=\{\t_1,\dots,\t_m\}\times[c_2n^{\e'},\infty)$, and $\mathcal K_3$ is the integral operator (on the same space) with kernel $K_3\equiv0$.
The first inequality in \eqref{Fredholm_inequalities} follows from the representation \eqref{Fredholm_representation} and $\Delta(\mathcal K_3)=1$. Using \eqref{Fredholm_comparison} and \eqref{comparison_kernels}, we find $\lv\Delta(\mathcal K_1)-\Delta(\mathcal K_2)\rv=\O(e^{-n^{\d'}})$ as $n\to\infty$ for some $\d'>0$. Now, \eqref{Fredholm_comparison} once more and the exponential decay of the extended Airy kernel give $\lv\Delta(\mathcal K_2)-\Delta(\mathcal K_3)\rv=\O(e^{-n^{\d'}})$ as $n\to\infty$, which yields the result for any $0<\d<\d'$.
\end{proof}

\section{The Pearcey Case: Proof of Theorem \ref{thm:main_Pearcey}}
With our understanding of the Airy case in Section \ref{sec_Airy} at hand, we can now give an efficient treatment of the Pearcey case. We follow the path outlined in Section \ref{sec_Airy}, highlighting the differences in the analysis without duplicating too many analogous arguments. {Throughout this section, $\e$ may be any number with $0<\e<\min\lb\frac{\k-3}{8(\k+1)},\frac1{24}\rb$. }
\subsection{Analysis of the Saddle Points}
We keep the same notation as in Section \ref{sec_Airy} with a number of adaptions to the Pearcey case. For instance, we will start with the analysis  of the saddle points of the functions $z\mapsto\phi_{n,\t_2}(z,v)$ and $w\mapsto-\phi_{n,\t_1}(w,u)$, where we now set
	\begin{align}\label{def:phi_Pearcey}
	\phi_{n,\t}(z,v):= \frac{n}{2t_n(\t)} \left(z-x_n(\t)-\frac{v}{c_3 n^{3/4}}\right)^2 + ng_{\mu_n}(z),
	\end{align}
	$t_n(\t)$ is short for $t_n^{\rm P}(\t)$ and $x_n(\t)$ stands for $x_n^{\rm P}(\t)$, see \eqref{def:tnP} and the equation below.
Necessary information on the location of the saddle points is given in the following lemma, which is a companion of Lemma \ref{lemma: saddle points}.

\begin{lemma}\label{lemma: saddle points_Pearcey}
	Suppose that Assumptions 1--3 hold, $G_2=0$ and $\kappa>3$. 
	\begin{enumerate}
		\item \label{expansionH2} We have uniformly for  $|z-x^*|= \mathcal{O}\left(n^{-\frac{1}{4}+\epsilon}\right)$ and $\t$ in compacts, 
		\[H_{t_{n}(\t),\mu_n}(z)=x_n(\t)+\frac{t_{n}(\t)G_3}{6} (z-x^*)^3+\mathcal O\left(n^{-\frac{3}{4}+\epsilon}\right),\] 
		as $n\to\infty$.
		\item We have
		\begin{align*} 
	z_n(\t_2,v)&=x^{*}+\mathcal{O}\left(n^{-1/4+\epsilon/3}\right), \\
		w_n(\t_1,u)&= x^{*}+\mathcal{O}\left(n^{-1/4+\epsilon/3}\right),
		\end{align*}
 where the implied constants are uniform for \(\vert u\vert,\lv v\rv\leq K n^{\epsilon}\) for any constant \(K >0\) and $\t_1,\t_2$ in compacts as $n\to\infty$.
	\end{enumerate}
\end{lemma}
\begin{proof} First we note that all $\O$-terms will be uniform for $\t$ in compacts as $n\to\infty$.
	
		For part (1), we note that by Proposition \ref{prop: Stieltjes comparison 2} (2), 
		we obtain	for $|z-x^*|= \mathcal{O}\left(n^{-\frac{1}{4}+\epsilon}\right)$
		\[H_{t_n(\t),\mu_n}(z) 
		=x_n(\t)+\left(1+t_{n}(\t) G_1\right)(z-x^{*})+\frac{t_{n}(\t)G_3}{6} (z-x^*)^3+\O\left(n^{-\frac{3}{4}-2\e}\right),	
		\]
		and this implies the result by $t_n(\t)G_1=-1+\O(n^{-1/2})$.

		For part (2), we use 
		Proposition \ref{prop: Stieltjes comparison 2} (2)  to conclude that with $x+iy=x^*+n^{-1/4+\epsilon}e^{i\theta}$ we have
		\begin{align*}
		\Im G_{\mu_n}(x+iy) =& \left(-\frac{1}{t_{n}(\tau)}+\mathcal{O}(n^{-1/2})\right)y +\frac{G_3}{6} y n^{-1/2+2\epsilon} \frac{\sin(3\theta)}{\sin(\theta)}+ \O\left(n^{-3/4-2\e}\right),
		\end{align*}
		as \(n \to \infty\). Hence, we obtain for large \(n\) using $G_3<0$
		\[-\frac{\Im G_{\mu_n}(x+iy)}{y} = \frac{1}{t_n(\tau)}-\frac{G_3\sin(3\theta)}{6\sin(\theta)} n^{-1/2+2\epsilon}\left(1+\mathcal{O}\left(n^{-\epsilon}\right)\right) > \frac{1}{t_n(\tau)},\]
		uniformly in \(\delta < \theta < \frac{\pi}{3}-\delta\) for any $0<\d<\pi/3$, a condition on $\d$ which we will assume from now on, and $n$ sufficiently large (which depends on $\d$). In a similar way, we obtain for angles \( \frac{2\pi}{3}+\delta < \theta < \pi-\delta\) and large \(n\)
		\[-\frac{\Im G_{\mu_n}(x+iy)}{y} > \frac{1}{t_n(\tau)}.\]
		For angles \( \frac{\pi}{3}+\delta < \theta < \frac{2\pi}{3} -\delta\) and large \(n\) on the other hand, we obtain
		\[-\frac{\Im G_{\mu_n}(x+iy)}{y} < \frac{1}{t_n(\tau)}.\]
		By definition of the function \(y_{t_n(\t),\mu_n}(x)\), this  implies that its graph 
		lies inside the circle of radius $n^{-1/4+\epsilon}$ around $x^*$ for \[x\in\left(x^*-n^{-1/4+\epsilon}\cos\left(\frac{\pi}{3}+\delta\right), x^*+n^{-1/4+\epsilon}\cos\left(\frac{\pi}{3}+\delta\right)
		\right)   
		\]
		and outside this circle for 
		\[x< x^*-n^{-1/4+\epsilon}\cos\left(\frac{\pi}{3}-\delta\right)\quad \mbox{and for}\quad x>x^*+n^{-1/4+\epsilon}\cos\left(\frac{\pi}{3}-\delta\right),\]
		if $n$ is sufficiently large.
		Thus we can define points 
				\[x_{\rm in,out}= H_{t_n(\t),\mu_n}\left(x^{*}+n^{-1/4 +\epsilon} e^{i \theta_{1,2}}\right) \in \mathbb{R},\]
		such that \[\pi/3-\delta<\theta_1<\pi/3+\delta,\qquad 2\pi/3-\delta<\theta_2<2\pi/3+\delta,\]
		and such that the part of the graph of $y_{t_n(\t),\mu_n}$ between those two points lies entirely inside the circle. The rest of the argument is similar as in the proof of Lemma \ref{lemma: saddle points}, and we omit the details.
	\end{proof}

\subsection{Choice of Contours}
The contour $\Sigma$ will be chosen as 
\begin{align}
\Sigma:=x^*+i\R.
\end{align}
For the choice of the contour $\Gamma$, recall from the proof of Lemma \ref{lemma: saddle points_Pearcey} that the graph of $x\mapsto y_{t_n(\t_1),\mu_n}(x)$ enters the disk around $x^*$ of radius $n^{-1/4+\e}$ (coming from the right) at a point $w_{1,n}$, about which we know that $\pi/3-\d<\arg(w_{1,n}-x^*)<\pi/3+\d$ for any $0<\d<\pi/3$, provided $n$ is sufficiently large. We will choose a fixed $0<\d<\frac{\pi}{24}$ in order to be able to use later on that $w_{1,n}$ lies in a sector where $w\mapsto\Re(w^4)$ is negative. Likewise, $y_{t_n(\t_1),\mu_n}$ leaves the disk at a point $w_{2,n}$, about which we know that $2\pi/3-\d<\arg(w_{2,n}-x^*)<2\pi/3+\d$ for $n$ large enough (with the same $\d$). Moreover, by construction, the part of the graph between the points $w_{1,n}$ and $w_{2,n}$ lies entirely inside the disk.

We now choose $\Gamma$ as depicted in Figure \ref{fig:contourGamma_P}.
\begin{figure}[h]
	\centering
	\footnotesize
	\def\svgwidth{0.7\columnwidth}
	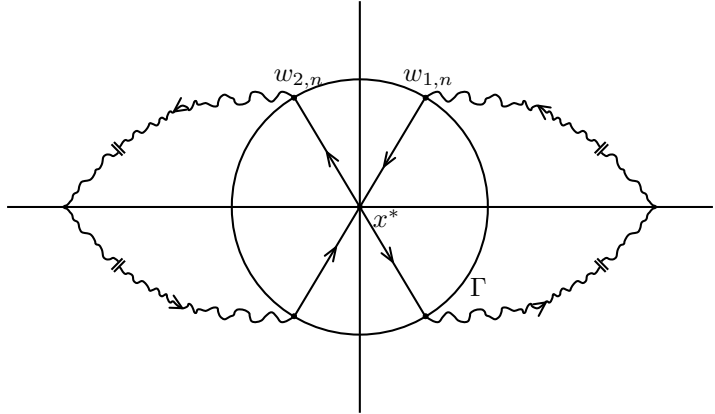
	\caption{Choice of contour $\Gamma$}\label{fig:contourGamma_P}
\end{figure}
Recall that $y_{t_n(\t_1),\mu_n}$ is of bounded support which contains all initial points $X_1(0),\dots,X_n(0)$. 
We first follow the graph of $y_{t_n(\t_1),\mu_n}$ from the right-most support point to $w_{1,n}$. From $w_{1,n}$ we take the straight line to $x^*$, then the straight line from $x^*$ to $w_{2,n}$ and from there we follow the graph of $y_{t_n(\t_1),\mu_n}$ to its left-most support point. We close the contour by joining it with its complex conjugate. We will, similarly to the Airy case, denote the part of $\Gamma$ inside of the disk as $\Gamma_{\rm in}$, i.e.
\begin{align}
\Gamma_{\rm in}:=[w_{1,n},x^*]\cup[x^*,w_{2,n}]\cup\overline{[w_{1,n},x^*]}\cup\overline{[x^*,w_{2,n}]}.
\end{align}
The remaining part of $\Gamma$ will be denoted by $\Gamma_{\rm out}$. Similarly, we define $\Sigma_{\rm in}:=[x^*-in^{-1/4+\e},x^*+in^{-1/4+\e}]$ and $\Sigma_{\rm out}:=\Sigma\setminus\Sigma_{\rm in}$.

From \eqref{def:Kn} and \eqref{conjugation}, and by a similar argument as for \eqref{rescaledkernel}, we have 
\begin{multline}
\frac{1}{c_3 n^{3/4}} \tilde{K}_{n,t_n(\t_1),t_n(\t_2)}\left(x_n(\t_1)+\frac{u}{c_3 n^{3/4}}, x_n(\t_2)+\frac{v}{c_3 n^{3/4}}\right)\\\label{rescaledkernel_Pearcey}
=  \frac{n^{1/4}e^{f_n(\t_2,v)-f_n(\t_1,u)}}{c_3 (2\pi i)^2 \sqrt{t_n(\t_1)t_n(\t_2)}} \int_\Sigma dz \int_{\Gamma} dw ~ \frac{e^{\phi_{n,\t_2} (z,v) - \phi_{n,\t_1} (w,u)}}{z-w}, 
\end{multline}
where $\phi_{n,\t}$ has been defined in \eqref{def:phi_Pearcey} and $f_n(\t,u)$ now stands for 
\[f_n\lb t_n(\t),x_n(\t)+\frac{u}{c_3n^{3/4}}\rb
\]
with $f_n$ from \eqref{gauge_factor}. }

We will again only consider the case of $\t_2\geq\t_1$, thus concentrating on the integral part of the correlation kernel in \eqref{def:Kn}, which we split into
\begin{align}
&\label{def:Kn1_P}{K}_{n,\t_1,\t_2}^{(1)}(u,v):=\frac{n^{1/4}e^{f_n(\t_2,v)-f_n(\t_1,u)}}{c_3 (2\pi i)^2 \sqrt{t_n(\t_1)t_n(\t_2)}} \int_{\Sigma_{\rm in}} dz \int_{\Gamma_{\rm in}} dw ~ \frac{e^{\phi_{n,\t_2} (z,v) - \phi_{n,\t_1} (w,u)}}{z-w},\\
&\label{def:Kn2_P}{K}_{n,\t_1,\t_2}^{(2)}(u,v):=\frac{n^{1/4}e^{f_n(\t_2,v)-f_n(\t_1,u)}}{c_3 (2\pi i)^2 \sqrt{t_n(\t_1)t_n(\t_2)}} \int_{\Sigma} dz \int_{\Gamma_{\rm out}} dw ~ \frac{e^{\phi_{n,\t_2} (z,v) - \phi_{n,\t_1} (w,u)}}{z-w},\\
&\label{def:Kn3_P}{K}_{n,\t_1,\t_2}^{(3)}(u,v):=\frac{n^{1/4}e^{f_n(\t_2,v)-f_n(\t_1,u)}}{c_3 (2\pi i)^2 \sqrt{t_n(\t_1)t_n(\t_2)}} \int_{\Sigma_{\rm out}} dz \int_{\Gamma_{\rm in}} dw ~ \frac{e^{\phi_{n,\t_2} (z,v) - \phi_{n,\t_1} (w,u)}}{z-w}.
\end{align}

\subsection{Proof of Theorem \ref{thm:main_Pearcey}}
\begin{proof}[Proof of Theorem \ref{thm:main_Pearcey}]
	As usual, all error terms will be uniform for $\t_1,\t_2$ in compacts as $n\to\infty$.
	We choose the parametrization
 \begin{align}
z=:x^{*}+\frac{c_3t_{\rm cr}\zeta}{n^{1/4}}\quad \text{ and } \quad w=:x^{*}+\frac{c_3t_{\rm cr}\w}{n^{1/4}}.\label{change_of_variables_Pearcey}
\end{align} This yields with Proposition \ref{prop: Stieltjes comparison 2} part (2) (recall $G_2=0$) for $z$ inside the disk
\begin{align}
&ng_{\mu_n}(z)=ng_{\mu_n}(x^*)+n^{3/4}G_0c_3t_{\rm cr}\zeta+\frac{n^{1/2}G_1c_3^2t_{\rm cr}^2\z^2}2+\frac{G_3c_3^4t_{\rm cr}^4\z^4}{24}+\O\lb\frac{\lv\zeta\rv}{n^{2\e}}\rb.\label{calculation:g_Pearcey}
\end{align} 
The analog of \eqref{calculation:phi} then reads 
\begin{align}
\phi_{n,\t_2}(z,v)=-\frac{\z^4}4-\frac{\z^2\t_2}2-\z v+\frac{nt_n(\t_2)G_0^2}2+\frac{n^{1/4}G_0v}{c_3}+ng_{\mu_n}(x^*)+\O\lb \frac{1}{n^{\e}}\rb \label{calculation:phi_P}
\end{align}
with the $\O$-terms in the two above equations uniform for $\lv\zeta\rv\leq Kn^{\e}$ and \(\lv v\rv\leq K\) for any absolute constant \(K>0\). Moreover, we have
\begin{align}
&\phi_{n,\t_2}(z,v)-\phi_{n,\t_1}(w,u)+f_n(\t_2,v)-f_n(\t_1,u)\\
&=-\frac{\z^4}4-\frac{\t_2\z^2}2-v\z+\frac{\w^4}4+\frac{\t_1\w^2}2+u\w +\O\lb \frac{1}{n^{\e}}\rb,\label{phase_function_local_P}
\end{align}
the $\O$-term being uniform for $\lv\zeta\rv,\lv \w\rv\leq Kn^{\e}$ and \(\lv v\rv,\lv u\rv\leq K\).
For the constant in front of the double integral we get
\begin{align}
\frac{t_{\rm cr}}{(2\pi i)^2\sqrt{t_n(\t_1)t_n(\t_2)}}=\frac{1}{(2\pi i)^2}+\O\lb \frac{ 1}{n^{1/2}}\rb.
\end{align}
Furthermore, there is a function $h_n(\t_1,\t_2,\z,\w,u,v)=\O\lb n^{-\e}\rb$ such that 
\begin{align}
&{K}_{n,\t_1,\t_2}^{(1)}(u,v)\\
&=\lb\frac{1}{(2\pi i)^2}+\O\lb \frac{ 1}{n^{1/2}}\rb\rb\int_{\widehat\Sigma} d\z \int_{\widehat\Gamma} d\w ~ \frac{e^{-\frac{\z^4}4-\frac{\t_2\z^2}2-v\z+\frac{\w^4}4+\frac{\t_1\w^2}2+u\w+h_n(\t_1,\t_2,\z,\w,u,v)}}{\z-\w},\label{K_n_1_beginning_P}
\end{align}
where $\widehat\Sigma:=\left[-\frac{in^{\e}}{c_3t_{\rm cr}},\frac{in^{\e}}{c_3t_{\rm cr}}\right]$ and
\begin{align}
\widehat\Gamma:=\left[\frac{(w_{1,n}-x^*)n^{1/4}}{c_3 t_{\rm cr}},0\right]\cup\left[0,\frac{(w_{2,n}-x^*)n^{1/4}}{c_3 t_{\rm cr}}\right]\cup\left[\frac{(\overline{w_{2,n}}-x^*)n^{1/4}}{c_3 t_{\rm cr}},0\right]\cup\left[0,\frac{(\overline{w_{1,n}}-x^*)n^{1/4}}{c_3 t_{\rm cr}}\right].
\end{align}
From here, we can show analogously to the proof of Proposition \ref{prop: int1} (from~\eqref{K_n^1_second_last} until the end of that proof) that
\begin{align}
{K}_{n,\t_1,\t_2}^{(1)}(u,v)=\mathbb K^{\rm P}_{\t_1,\t_2}(u,v)+\O(n^{-\e})
\end{align}
with the $\O$-term being uniform for $u,v\in[-M,M]$ for any fixed $M>0$. Here we use in particular that by our choice of $\e,\d$, the points $\frac{(w_{j,n}-x^*)n^{1/4}}{c_3 t_{\rm cr}},\ j=1,2$, and their complex conjugates lie in the sectors where $w\mapsto\Re w^4$ is negative, allowing to extend and deform the contour $\widehat \Gamma$ to $\Gamma^{\rm P}$ from Theorem \ref{thm:main_Pearcey} at the expense of an exponentially small error.

To estimate $K_{n,\t_1,\t_2}^{(j)}$, $j=2,3$, we first note that in analogy to Lemma \ref{lemma:auxiliaryremaining} we have for some $d>0$ and $z\in \Sigma,\ w\in\Gamma_{\rm out}$ 
\begin{align}\label{remaining1_P}
\Re\left(\phi_{n,\t_2}(z,v)-\phi_{n,\t_1}(w,u)+f_n(\t_2,v)-f_n(\t_1,u)\right)\leq -dn^{4\e}.
\end{align}
  For $z\in\Sigma_{\rm out}\cup\overline{\Sigma_{\rm out}},\ w\in\Gamma_{\rm in}\cup\overline{\Gamma_{\rm in}}$ we have 
\begin{align}\label{eq1_lemma_decay_P}
\Re\left(\phi_{n,\t_2}(z,v)-\phi_{n,\t_1}(w,u)+f_n(\t_2,v)-f_n(\t_1,u)\right)\leq -dn^{4\e}.
\end{align}
The proof of \eqref{remaining1_P} and \eqref{eq1_lemma_decay_P} is analogous to the proof of Lemma \ref{lemma:auxiliaryremaining}. For instance, for \eqref{remaining1_P} it suffices to note that by \eqref{phase_function_local_P}, \eqref{remaining1_P} holds for $z\in\Sigma_{\rm in}$ and $w=w_{i,n},\,i=1,2$. From there it can be extended to $\Sigma_{\rm out}$ using that, according  to Lemma \ref{lemma:realpartphi}, the real part of $\phi_{n,\t_2}(z,v)$ decreases with $\lv z\rv\to\infty$, $z\in\Sigma_{\rm out}$. Here we use that the graph of $y_{t_n(\t_2),\mu_n}(x)$ lies in the disk between $w_{1,n}$ and $w_{2,n}$ and thus $\Sigma_{\rm out}$ lies above $y_{t_n(\t_2),\mu_n}$ in the upper half plane and below $\overline{y_{t_n(\t_2),\mu_n}}$ in the lower half plane, respectively. The extension from $w_{i,n}$ to $w\in\Gamma_{\rm out}$ uses the increase of $\phi_{n,\t_1}(w,u)$ along $y_{t_n(\t_1),\mu_n}$ (and $\overline{y_{t_n(\t_1),\mu_n}}$) as $w$ moves further away from the saddle point $w_n$, which lies inside the disk by Lemma \ref{lemma: saddle points_Pearcey}. Proving \eqref{eq1_lemma_decay_P} is analogous.

To finish the proof, it suffices to show for some $D>0$
$${K}_{n,\t_1,\t_2}^{(j)}(u,v)=\mathcal O(e^{-n^{D}}),\quad j=2,3,$$
uniformly for $\tau_1,\t_2$ in compact subsets of $\mathbb R$, and for $u,v\in [-M, M]$ for any $M>0$. This is however fully analogous to the proof of Lemma \ref{lemma: int2}.

\end{proof}

\appendix
\section{Space-time correlation functions}\label{appendix:correlations}

In this appendix, we present some background on the space-time correlation functions  {of} the NIBM process.

 We start by recalling that {by definition,} the transition density of the Markov process $(X(t))_t$ can be obtained as the joint probability density function of the eigenvalues of 
\begin{align}
M\mapsto Z_n^{-1}e^{-\frac{n}{2t}\Tr(M(0)-M)^2},
\end{align}
where $Z_n$ is the normalization constant. It is well-known \cite{BH3,BH2,Johansson01} that employing the Harish-Chandra/Itzykson-Zuber formula, the transition density $\x^{(1)}\mapsto p_t(\x^{(0)};\x^{(1)})$ of $(X(t))_t$ (given $X(0)$) can then be computed as
\begin{align}
p_t(\x^{(0)};\x^{(1)})=\lb\frac{n}{2\pi t}\rb^{\frac n2}\frac{\prod_{i<j} (x^{(1)}_j-x_i^{(1)})}{\prod_{i<j} (x^{(0)}_j-x_i^{(0)})}\det\lb e^{-\frac{n}{2t}(x_i^{(0)}-x_j^{(1)})^2}\rb_{1\leq i,j\leq n},
\end{align}
where $\x^{(i)}=(x^{(i)}_1,\dots,x_n^{(i)})\in W_n:=\{\x\in\R^n:x_1\leq\ldots\leq x_n\},\, i=0,1$ and $\x^{(0)}:=X(0)$. A priori, this density is only defined for distinct initial values $x_j^{(0)}$ but this condition is readily removed by invoking continuity arguments. By the Chapman-Kolmogorov equations, the finite-dimensional distributions of $(X(t))_t$, i.e.~the joint distributions of all finite collections of vectors \(X(t_1),\dots,X(t_k)\in\R^{n}\) for any choice $0=t_0<t_1<\ldots<t_k$, $k\in\mathbb N$, have the densities
\begin{align*}
&W_n^k\ni(\x^{(1)},\dots,\x^{(k)})\mapsto p_{t_1\dots,t_k}(\x^{(0)};\x^{(1)},\dots,\x^{(k)})\\
&:=p_{t_1}(\x^{(0)};\x^{(1)})p_{t_2-t_1}(\x^{(1)};\x^{(2)})\dots p_{t_{k}-t_{k-1}}(\x^{(k-1)};\x^{(k)})\\
&=\lb\prod_{l=1}^k\frac{n}{2\pi (t_l-t_{l-1})}\rb^{\frac {n}2}\frac{\prod_{i<j} (x^{(k)}_j-x_i^{(k)})}{\prod_{i<j} (x^{(0)}_j-x_i^{(0)})}\prod_{l=1}^k\det\lb e^{-\frac{n}{2(t_l-t_{l-1})}(x^{(l-1)}_i-x^{(l)}_j)^2}\rb_{1\leq i,j\leq n}.
\end{align*}

An important property of NIBM is the determinantality of its correlation functions. In order to properly introduce these functions, we consider the symmetrized density $\hat p_{t_1\dots,t_k}(\x^{(0)};\x^{(1)},\dots,\x^{(k)})$ on \(\left(\R^{n}\right)^k\), defined for $\x^{(1)},\dots,\x^{(k)}\in\R^n$ 
\begin{align}
\hat p_{t_1\dots,t_k}(\x^{(0)};\x^{(1)},\dots,\x^{(k)}):=\frac{1}{(n!)^k} p_{t_1\dots,t_k}(\x^{(0)};\x^{(1)}_\leq,\dots,\x^{(k)}_\leq),\label{multi-time-density}
\end{align}
where for $\x^{(i)}\in\R^n$, $\x^{(i)}_\leq$ denotes the unique permuted vector built from $\x^{(i)}$ that lies in $W_n$. 
Now, the space-time correlation functions are for any collection of integer numbers $1\leq m_1,\dots,m_k\leq n$ defined as
\begin{multline}\label{def:multitime}
\rho^{(n)}_{t_1,\dots,t_k}(\x^{(1)}_{m_1},\dots,\x^{(k)}_{m_k}):=\frac{(n!)^k}{\prod_{l=1}^k(n-m_l)!}\\
\times \int\limits_{\R^{kn-\sum_{j=1}^km_k}} \hat p_{t_1,\dots,t_k}(\x^{(1)},\dots,\x^{(k)})dx^{(1)}_{m_1+1}\dots dx^{(1)}_{n}\dots dx^{(k)}_{m_k+1}\dots dx^{(k)}_{n}.
\end{multline}
{
The correlation functions are multiples of the marginal densities of $\hat p_{t_1\dots,t_k}(\x^{(0)};\cdot,\dots,\cdot)$ 
and allow to express the expectation of statistics in terms of integrals. 
 Note that due to the symmetrization, the correlation functions do not directly describe individual paths of the process, i.e.~$(x_1,x_2)\mapsto\rho^{(n)}_{t}(x_1,x_2)$ does not describe the correlations of the two smallest paths but rather the correlations of paths around the point $x_1$ and paths around $x_2$ (at time $t$). However, statistics of special paths, e.g.~those with gaps to one side, admit an effective representation in terms of correlation functions{, as for instance used in the proof of Theorem \ref{corollary_TW}}. Alternatively, one may interpret the correlation functions as joint intensities of the time-dependent point process $\sum_{j=1}^n\d_{X_j(t)}$. }
 
{
The structure of the density \eqref{multi-time-density} being a product of $k+2$ determinants in the variables $\x^{(0)},\dots,\x^{(k)}$ allows for an application of the Eynard-Mehta theorem \cite{EynardMehta,BorodinRains} that yields the determinantality of the correlation functions, i.e.~the ability to express the correlation functions as determinants of matrices given by a certain kernel function. {This leads to the formulas} \eqref{determinantality}--\eqref{def:Kn}.
}

\printbibliography

\end{document}